\documentclass[10pt,reqno]{amsart} 
\usepackage{epsfig,amssymb,amsmath,version}
\usepackage{amssymb,version,graphicx,fancybox,mathrsfs}
\usepackage{url,hyperref}
\usepackage{subfigure}
\usepackage{color}
\usepackage{stmaryrd}
\usepackage{multirow}
\usepackage{booktabs,siunitx}
\usepackage{booktabs}
\usepackage{multicol}
\usepackage{booktabs} 
\usepackage{amsmath}
\usepackage{setspace}
\usepackage{array}
\usepackage{placeins}
\usepackage{colortbl}
\usepackage{enumerate}
\usepackage{enumitem}
\usepackage{tabu}                     
\usepackage{multirow}                 
\usepackage{multicol}                 
\usepackage{multirow}                
\usepackage{float}                    
\usepackage{makecell}                 
\usepackage{booktabs}                 
\usepackage{diagbox,bm}
\usepackage{algorithm,float}
\usepackage{algorithmicx}
\usepackage{algpseudocode}
\usepackage{lipsum}
\usepackage{changepage}

\floatname{algorithm}{Algorithm 1}

\catcode`\@=11 \theoremstyle{plain}
\@addtoreset{equation}{section}   

\@addtoreset{figure}{section}
\renewcommand\thefigure{\thesection.\@arabic\c@figure}
\newtheorem{thm}{\bf Theorem}

\newtheorem{cor}{\bf Corollary}

\newtheorem{lmm}{\bf Lemma}

\theoremstyle{remark}

\definecolor{ligreen}{rgb}{0.0, 0.3, 0.0}

\definecolor{darkblue}{rgb}{0.0, 0.0, 0.55}

\definecolor{anti-flashwhite}{rgb}{0.55, 0.57, 0.68}

\newcommand{\bs}[1]{\boldsymbol{#1}}

\graphicspath{{./Figures/} } 

\makeatletter
\newenvironment{breakablealgorithm}
  {
   \begin{center}
     \refstepcounter{algorithm}
     \hrule height.8pt depth0pt \kern2pt
     \renewcommand{\caption}[2][\relax]{
       {\raggedright\textbf{\ALG@name~\thealgorithm} ##2\par}%
       \ifx\relax##1\relax 
         \addcontentsline{loa}{algorithm}{\protect\numberline{\thealgorithm}##2}%
       \else 
         \addcontentsline{loa}{algorithm}{\protect\numberline{\thealgorithm}##1}%
       \fi
       \kern2pt\hrule\kern2pt
     }
  }{
     \kern2pt\hrule\relax
   \end{center}
  }
\makeatother

\begin{document}
\bibliographystyle{plain}

\title[A novel algorithm tailored for unconstrained optimization problems] {A novel numerical method tailored for unconstrained optimization problems}
\author[
	L. Li,\,    P. Xie\,  $\&$\,  L. Zhang
	]{
		\;\; Lin Li${}^{\dag}$,   \;\;  Pengcheng Xie${}^{\ddag}$ \;\; and\;\; Li Zhang${}^{\S}$
		}
	\thanks{${}^{\dag}$School of Mathematics and Physics,  University of South China, Hengyang, 421001, China. Email: lilinmath@usc.edu.cn (L. Li).\\
			\indent ${}^{\ddag}$Corresponding author.  Applied Mathematics and Computational Research Division, Lawrence Berkeley National Laboratory, 1 Cyclotron Road, Berkeley, 94720, CA, USA. Email: pxie@lbl.gov, pxie98@gmail.com (P. Xie). The work was done before the corresponding author joined LBNL. \\
		\indent ${}^{\S}$School of Mathematics, Southwestern University of Finance and Economics, Chengdu, China. Email: zhanglily@swufe.edu.cn (L. Zhang)}

\keywords{Nonsmooth, quadratic model, unconstrained optimization} \subjclass[2000]{90C56, 90C30, 65K05, 90C90}

\begin{abstract}
Unconstrained optimization problems become more common in scientific computing and engineering applications with the rapid development of artificial intelligence, and numerical methods for solving them more quickly and efficiently have been getting more attention and research. Moreover, an efficient method to minimize all kinds of objective functions is urgently needed, especially the nonsmooth objective function. Therefore, in the current paper, we focus on proposing a novel numerical method tailored for unconstrained optimization problems whether the objective function is smooth or not. To be specific, based on the variational procedure to refine the gradient and Hessian matrix approximations, an efficient quadratic model with $2n$ constrained conditions is established. Moreover, to improve the computational efficiency, a simplified model with 2 constrained conditions is also proposed, where the gradient and Hessian matrix can be explicitly updated, and the corresponding boundedness of the remaining $2n-2$ constrained conditions is derived. On the other hand, the novel numerical method is summarized in the Algorithm \ref{algorithm2025011301}, 
and approximation results on derivative information are also analyzed and shown. Numerical experiments involving smooth, derivative blasting, and non-smooth problems are tested, demonstrating its feasibility and efficiency. Compared with existing methods, our proposed method can efficiently solve smooth and non-smooth unconstrained optimization problems for the first time, and it is very easy to program the code, indicating that our proposed method not also has great application prospects, but is also very meaningful to explore practical complex engineering and scientific problems.
\end{abstract}
\maketitle

\section{Introduction}\label{sect1}

Nowadays, optimization problems are more and more common in nature optimizes, physical (or chemical) systems, engineering design, and so on. It is worth mentioning that for the most popular machine learning in computing differential equations, minimizing a loss function is its core content, and the corresponding optimization algorithm plays a key role in achieving a minimal solution of the loss function. Furtherly, many practical problems can be viewed as an
unconstrained optimization problems, where an objective function that depends on real variables
often needs to be optimized, and the corresponding optimization mathematical formulation is
\begin{equation}\label{eq2024121301}
\min_{{\bs x}\in \mathbb{R}^n} f({\bs x}),
\end{equation}
where $f: \mathbb{R}^n \rightarrow \mathbb{R}$ is a real function. So far there also exist many methods to solve \eqref{eq2024121301}. Based on the regularity of the objective function $f$, these methods can be divided roughly into two categories. In the current paper, we will propose a new and efficient method for solving \eqref{eq2024121301} whether the objective function $f$ is smooth or not, indicating that the proposed method is suitable for more general unconstrained optimization problem \eqref{eq2024121301}. Before we introduce our approach, we feel compelled to elaborate on existing methods and motivations.

When the objective function $f$ is smooth with respect to ${\bs x}$, some methods requiring derivative information have been proposed and developed. Here the first thing to mention is the line search method. To be specific, in the line search method, the main idea is to choose a direction ${\bs d}_{k}$ and search along this direction from the current iterate point ${\bs x}_{k}$, aiming to find a new iterate with a lower function value. In other words, an unconstrained minimization model to find a step length $\alpha$ along the direction ${\bs d}_k$ should be solved, i.e.,
\begin{equation}\label{eq2024121701}
\min_{\alpha > 0}f({\bs x}_k + \alpha {\bs d}_{k}).
\end{equation}
It is worth pointing out that when \eqref{eq2024121701} is solved exactly, the computational cost may be expensive. Instead, the line search method will generate a limited number of trial step lengths until it finds one that loosely approximates the minimum of \eqref{eq2024121701}. As seen in \cite{nocedal2006numerical}, a popular inexact line search condition is that $\alpha_k$ should provide sufficient decrease, as measured by the following inequality:
\begin{equation}\label{eq2024121702}
f({\bs x}_k + \alpha {\bs d}_k) \leq f({\bs x}_k) + c_{1}\alpha \nabla f^{\top}_k {\bs d}_k,
\end{equation}
for some constant $c_1\in (0, 1)$, where the reduction in $f$ is proportional to both the directional derivative $\nabla f^{\top}_k {\bs d}_k$ and the step length $\alpha_k$. The inequality \eqref{eq2024121702} is called the \textit{Armijo condition}. However, it is worth pointing out that the \textit{Armijo condition}   is not enough by itself to ensure that the line search method makes reasonable progress. To rule out unacceptably short steps, a second requirement is introduced, where the step length $\alpha_k$ is to satisfy
\begin{equation}\label{eq2024121703}
\nabla f({\bs x}_k + \alpha_k {\bs d}_k)^{\top} {\bs d}_k \geq c_2 \nabla f^{\top}_{k}{\bs d}_k,
\end{equation}
for some constant $c_2\in (c_1, 1)$. Obviously, \eqref{eq2024121703} ensures that the slope of $f$ at the point ${\bs x}_k + \alpha_k{\bs d}_k$ is greater than $c_2$ times the initial slope ($\alpha = 0$), which provide an indication that $f$ can be reduced significantly by moving further along the chosen direction. \eqref{eq2024121702} and \eqref{eq2024121703} are also known collectively as the \textit{Wolfe conditions}. Similar to the \textit{Wolfe conditions}, the \textit{Goldstein conditions} has also been proposed to ensure that the step length $\alpha$ achieves sufficient decrease but is not too short, and its mathematical formulation is
\begin{equation}
f({\bs x}_k) + (1 - c)\alpha_k \nabla f^{\top}_k {\bs d}_k \leq f({\bs x}_k + \alpha_k{\bs d}_k) \leq f({\bs x}_k) + c\alpha_k f^{\top}_k {\bs d}_k
\end{equation}
with $c\in (0, 1/2)$. In addition, the convergence of line search methods can be seen in \cite{nocedal2006numerical, sun2006optimization} for more details. When ${\bs d}_k$ in \eqref{eq2024121701} is selected as $-\nabla f_{k}$ (i.e., the steepest descent direction) at each step, the corresponding line search method is also called the steepest descent method. Unfortunately, two successive steepest descent directions in the steepest descent method are orthogonal. This leads to appear the zigzagging near the solution, which seriously affects its convergence. To improve the convergence, the famous Newton method has been proposed, i.e., ${\bs d}_k := -(\nabla^2 f_k)^{-1}\nabla f_k$ in \eqref{eq2024121701} is chosen. As we know, the Newton method depends heavily on the initial guess, and requires that Jthe acobian matrix has full column rank, which greatly reduces the computational efficiency. As a result, some Quasi-Newton methods have been proposed and developed, e.g., \cite{1977More}. When ${\bs d}_k$ in \eqref{eq2024121701} is selected by using a linear combination of ${\bs d}_{k-1}$ and $-\nabla f_{k}$, the corresponding line search method is called the Conjugate Gradient method, e.g., see \cite{nocedal2006numerical}. It is worth pointing out that for line search methods above, a search direction is generated, and then a suitable step length $\alpha$ along this direction is found. On the other hand, the trust-region method presented in \cite{sun2006optimization} was proposed to choose the direction and the step length simultaneously. Moreover, its main idea is that a trust region around the current point is defined, where a model to be an adequate representation of the objective function is used, and then a step will be chosen to approximate the minimizer of the model in the trust region. Recently, as studied in \cite{2022Two}, the trust-region method has some advantages compared with the classical Newtonian method, where the choices of initial guesses are much more relaxed and fairly flexible at times, and it is easier to find new solutions more efficiently and more quickly. Moreover, some interesting recent works on its further application in computing nonlinear differential equations can be seen in \cite{LiYe2024,li2025spectrallevenbergmarquardtdeflationmethodmultiple}.

When the objective function $f$ is non-smooth with respect to ${\bs x}$, solving \eqref{eq2024121301} will encounter inherent difficulties and are truly challenging. However, to our best knowledge, so far there are some optimization methods without derivative information (i.e., derivative-free methods), and they can be divided into three categories. Based on sampling points from geometric patterns, the first is a class of direct search methods to the variable space, where a relatively large number of function evaluations need to be computed. This severely reduces their computational efficiency. The second class of methods presented in \cite{1977More} is based on the finite difference and the Quasi-Newton method, where finite differences used to approximate derivative information may not always be robust, and its accuracy is hard to satisfy. These greatly affect their application. In \cite{2009Connn, 2002Marazzi, 2004Powell, 2009JJMore}, the third class of methods has been developed, where the sequential minimizations of models are constructed to approximate the objective function. Moreover, most of the third class of methods are based on linear (or quadratic) approximations \cite{powell2009uobyqa, powell2006newuoa}. In 2004, Powell proposed a underdetermined interpolating model with $2n+1$ sampling points \cite{2004Powell}. Later, in \cite{conn2000trust, nocedal2006numerical}, a trust region model-based method was proposed to ensure a good global convergence. Moreover, more interesting developments can be seen in \cite{Porcelli,LedWild2015,custodio2010incorporating,hare2025expected, Dzahini2024, 2014Zhangzaikun, 2012ZZhang, 2010Zhang, LarsonMenickellyWild2019, 2024Fan,xie2023twodimensional, xie2023derivativecombine,xie2025remuregionalminimalupdating}. On the other hand, some derivative-free solvers have been developed, e.g., PDFO \cite{ragonneau2024pdfo}, ORBIT \cite{wild2008orbit}, CONORBIT \cite{ling2010conorbit}, BOOSTERS \cite{diouane2020boosters} and DFLS \cite{powell2014dfls}.

In terms of the methods mentioned above, the Newton method shows the fast convergence near the solution, and is easy to program and implement. In addition, in the third class of methods above, a quadratic model to approximate a non-smooth objection function $f$ is very useful to overcome the lack of derivative information. Therefore, here we
will take full account of these advantages, and
a novel numerical method tailored for the unconstrained optimization problem \eqref{eq2024121301} is proposed whether the objective function $f$ is smooth or not. To be specific, $2n$ constrained conditions for determining a quadratic model to approximate the objective function is first proposed, where the variational procedure to refine the gradient and Hessian matrix approximations is introduced to ensure an effective representation of the objective function's local behavior, improving  the accuracy of quadratic approximation. Moreover, a simplified model with 2 constrained conditions is also proposed for solving \eqref{eq2024121301}, which greatly reduces the computational complexity.

The remainder of this paper is as follows. In Section \ref{sect2}, a constrained quadratic model approximated the objective function is established, where $2n$ constrained conditions are derived and presented in detail. In addition, a new simplified model with 2 constrained conditions is also proposed for improving the computational efficiency, and the corresponding boundedness of the remaining $2n-2$ constrained conditions is also derived. In addition, approximation results on derivative information are shown. In Section \ref{sect3}, based on the results presented in Section \ref{sect2}, a novel numerical method for computing \eqref{eq2024121301} is summarized and presented. To validate the effectiveness and feasibility of our proposed method, numerical experiments are tested in Section \ref{sect5}. Finally, Section \ref{sect6} summarizes our results and outlines potential future applications.

\section{An efficient constrained quadratic model and its simplification}\label{sect2}

In this section, to address the general (or complicated) objective function $f$ in \eqref{eq2024121301}, we mainly focus on constructing a smooth updating quadratic model to approximate it within a local range, and a simplified unconstrained optimization model is derived for the first time. In addition, the corresponding bounded analysis and approximation results on derivation information are also considered and shown.

Firstly, for the $k$-th iteration, a quadratic model $Q^{(k)}({\bs x})$ is introduced to approximate $f$, i.e.,
\begin{equation}\label{eq2024122301}
Q^{(k)}({\bs x}) = c +({\bs g}^{(k)})^{\top}{\bs x} - \frac{1}{2} {\bs x}^{\top} {\bs G}^{(k)} {\bs x}\in \mathbb{P}_2,
\end{equation}
where ${\bs G}^{(k)}\in {\mathbb R}^{n\times n}, {\bs g}^{(k)}\in {\mathbb R}^{n}$ and $c \in {\mathbb R}$ are unknown variables to be determined, and $\mathbb{P}_2$ represents the set of all algebraic polynomials of degree $\leq 2$. In the current paper, since we are based on the Newton iteration, the unknown variable $c$ can be ignored. Therefore, in the next section we will present the detailed process for determining ${\bs g}^{(k)}$ and ${\bs G}^{(k)}$ in \eqref{eq2024122301}.

In fact, here the quadratic model $Q^{(k)}({\bs x})$ should be constructed to replace the objective function $f$ for each minimization iteration. This means that ${\bs g}^{(k)}$ and ${\bs G}^{(k)}$ in \eqref{eq2024122301} need to be iteratively updated. To reduce the computational burden, we consider $({\bs G}^{(k)})^{\top} = {\bs G}^{(k)}$. Moreover, $\{{\bs x}_i\}_{i = k-n+1}^{k}$ are assumed to be model points that determine ${\bs g}^{(k)}$ and ${\bs G}^{(k)}$. In addition, we denote\vspace{-0.25cm}
\begin{equation}\label{eq2024122302}
{\bs \sigma}_i := {\bs x}_i - {\bs x}_{i-1}   \quad\quad\;\;\,\, \quad\quad   {\bs \tau}_{i} := {\bs x}_i - {\bs x}_0 = \sum_{j=1}^{i}{\bs \sigma}_{j},
\end{equation}\vspace{-0.3cm}
\begin{equation}\label{eq2024122303}
\triangle {\bs g}^{(k)} := {\bs g}^{(k)}_{k} - {\bs g}^{(k-1)}_{k-1} \quad\quad\quad    \triangle {\bs G}^{(k)} := {\bs G}^{(k)} - {\bs G}^{(k-1)},
\end{equation}
where ${\bs g}^{(k)}_i$ and ${\bs G}^{(k)}_i$ represent the gradient  and the Hessian matrix at the model point ${\bs x}_i$ given by the $k$-th quadratic model $Q^{(k)}({\bs x})$.
Furthermore, with \eqref{eq2024122302}-\eqref{eq2024122303}, the following relationship can be derived easily, i.e.,
\begin{equation}\label{eq2024122701}
\begin{split}
{\bs g}^{(k)}_{i} = &\;\, {\bs g}^{(k)}_{i-1} + \triangle{\bs G}^{(k)}{\bs \sigma}_i\\
=  &\;\, {\bs g}^{(k)}_{i-1} - {\bs g}^{(k)}_{k} +{\bs g}^{(k)}_{k} + \triangle{\bs G}^{(k)}{\bs \sigma}_i\\
 =   &\;\, {\bs G}^{(k)}({\bs x}_{i-1} - {\bs x}_k) + {\bs g}^{(k-1)}_{k-1} + \triangle{\bs g}^{(k)} + \triangle{\bs G}^{(k)}{\bs \sigma}_i\\
 =   &\;\, {\bs G}^{(k)}({\bs \tau}_i - {\bs \tau}_k - {\bs \sigma}_i) + {\bs g}^{(k-1)}_{k-1} + \triangle{\bs g}^{(k)} + \triangle{\bs G}^{(k)}({\bs \tau}_i - {\bs \tau}_{i-1})\\
 =  &\;\, ({\bs G}^{(k-1)} + \triangle{\bs G}^{(k)})({\bs \tau}_i - {\bs \tau}_k - {\bs \sigma}_i) + {\bs g}^{(k-1)}_{k-1} + \triangle{\bs g}^{(k)} + \triangle{\bs G}^{(k)}({\bs \tau}_i - {\bs \tau}_{i-1})\\
 =  &\;\, {\bs g}^{(k-1)}_{k-1} + \triangle{\bs g}^{(k)} - {\bs G}^{(k-1)}({\bs \tau}_{k} - {\bs \tau}_{i-1}) - \triangle{\bs G}^{(k)}({\bs \tau}_{k} - {\bs \tau}_{i}).
\end{split}
\end{equation}
Motivated by \cite{powell04onupdating,XIE2025116146}, here we request that for the model points $\{{\bs x}_i\}_{i = k-n+1}^{k}$, the gradient ${\bs g}^{(k)}_{i}$ is orthogonal to the descent direction at the previous step (i.e., ${\bs \sigma}_i$ in \eqref{eq2024122302}), i.e.,
\begin{equation}\label{eq2024122401}
{\bs \sigma}^{\top}_i {\bs g}^{(k)}_i = {\bs \sigma}^{\top}_i ({\bs g}^{(k-1)}_{k-1} + \triangle {\bs g}^{(k)} - {\bs G}^{(k)}({\bs \tau}_k - {\bs \tau}_i)) = 0,   \quad\;  k-n+1 \leq i \leq k.
\end{equation}
From \eqref{eq2024122401}, we can derive that
\begin{equation}\label{eq2024122402}
{\bs \sigma}^{\top}_{i}\triangle {\bs g}^{(k)} = {\bs \sigma}^{\top}_{i}{\bs G}^{(k)}({\bs \tau}_k - {\bs \tau}_i) - {\bs \sigma}^{\top}_{i}{\bs g}^{(k-1)}_{k-1}.
\end{equation}
With \eqref{eq2024122303}, \eqref{eq2024122402} becomes
\begin{equation}\label{eq2024122403}
{\bs \sigma}^{\top}_{i}\triangle {\bs g}^{(k)} - {\bs \sigma}^{\top}_{i} \triangle{\bs G}^{(k)}({\bs \tau}_k - {\bs \tau}_i) = {\bs \epsilon}^{(k)}_{i}, \quad\; k-n+1 \leq i \leq k,
\end{equation}
where
\begin{equation}\label{eq2024122407}
{\bs \epsilon}^{(k)}_{i} = {\bs \sigma}^{\top}_{i}{\bs G}^{(k-1)}({\bs \tau}_k - {\bs \tau}_i) - {\bs \sigma}^{\top}_{i}{\bs g}^{(k-1)}_{k-1}.
\end{equation}
On the other hand, to ignore $c$ in \eqref{eq2024122301},
we assume that
\begin{equation}\label{eq2025020902}
f({\bs x}_i) = Q^{(k)}({\bs x}_i)  \quad\;   k-n+1 \leq i \leq k.
\end{equation}
Let $f_i := f({\bs x}_i)$\, $(k-n+1 \leq i \leq k)$, and we can derive
\begin{equation}\label{eq2024122404}
\begin{split}
\triangle f_{i}=  & \,f_i - f_{i-1}\\
~   = &\, Q^{(k)}_{i} - Q^{(k)}_{i-1}\\
~   = &\, ({\bs x}_{i-1}+{\bs \sigma}_i)^{\top}{\bs g}^{(k)}_{i} - {\bs x}^{\top}_{i-1}{\bs g}^{(k)}_{i-1} + \frac{1}{2}{\bs x}^{\top}_{i-1}{\bs G}^{(k)}{\bs x}_{i-1}\\
~    & \, -\frac{1}{2}({\bs x}_{i-1}+{\bs \sigma}_i)^{\top}{\bs G}^{(k)}({\bs x}_{i-1}+{\bs \sigma}_i).
\end{split}
\end{equation}
Substituting \eqref{eq2024122701} into \eqref{eq2024122404} holds
\begin{equation}\label{eq2024122702}
\begin{split}
\triangle f_{i}  = &\, ({\bs x}_{i-1}+{\bs \sigma}_i)^{\top}({\bs g}^{(k)}_{i-1}+{\bs G}^{(k)}{\bs \sigma}_i) + \frac{1}{2}{\bs x}^{\top}_{i-1}{\bs G}^{(k)}{\bs x}_{i-1}\\
~  &\, -\frac{1}{2}({\bs x}_{i-1}+{\bs \sigma}_i)^{\top}{\bs G}^{(k)}({\bs x}_{i-1}+{\bs \sigma}_i) - {\bs x}^{\top}_{i-1}{\bs g}^{(k)}_{i-1}\\
~  = &\,  {\bs \sigma}^{\top}_{i}{\bs g}^{(k)}_{i-1} + \frac{1}{2}{\bs \sigma}^{\top}_{i}{\bs G}^{(k)}{\bs \sigma}_{i}.
\end{split}
\end{equation}
With \eqref{eq2024122303}, \eqref{eq2024122702} becomes
\begin{equation}\label{eq2024122405}
{\bs \sigma}^{\top}_{i}\triangle {\bs g}^{(k)} - {\bs \sigma}^{\top}_{i}\triangle {\bs G}^{(k)}({\bs \tau}_k - {\bs \tau}_{i-1}) + \frac{1}{2}{\bs \sigma}^{\top}_{i}\triangle {\bs G}^{(k)}{\bs \sigma}_{i} = {\bs \rho}^{(k)}_i,  \quad\; k-n+1 \leq i \leq k,
\end{equation}
where
\begin{equation}\label{eq2024122406}
{\bs \rho}^{(k)}_{i} = \triangle f_{i} - {\bs \sigma}^{\top}_{i}{\bs g}^{(k-1)}_{i-1} + {\bs \sigma}^{\top}_{i}{\bs G}^{(k-1)}({\bs \tau}_k - {\bs \tau}_{i-1}) - \frac{1}{2}{\bs \sigma}^{\top}_{i}{\bs G}^{(k-1)}{\bs \sigma}_{i}.
\end{equation}
By using \eqref{eq2024122302}, \eqref{eq2024122405} and \eqref{eq2024122406} can be transformed as follows:
\begin{equation}\label{eq20241224051}
{\bs \sigma}^{\top}_{i}\triangle {\bs g}^{(k)} - {\bs \sigma}^{\top}_{i}\triangle {\bs G}^{(k)}({\bs \tau}_k - {\bs \tau}_{i}) - \frac{1}{2}{\bs \sigma}^{\top}_{i}\triangle {\bs G}^{(k)}{\bs \sigma}_{i} = {\bs \rho}^{(k)}_i
\end{equation}
and
\begin{equation}\label{eq20241224061}
{\bs \rho}^{(k)}_{i} = \triangle f_{i} - {\bs \sigma}^{\top}_{i}{\bs g}^{(k-1)}_{i-1} + {\bs \sigma}^{\top}_{i}{\bs G}^{(k-1)}({\bs \tau}_k - {\bs \tau}_{i}) + \frac{1}{2}{\bs \sigma}^{\top}_{i}{\bs G}^{(k-1)}{\bs \sigma}_{i}.
\end{equation}
With \eqref{eq2024122407} and \eqref{eq2024122406}, we denote that
\begin{equation}\label{eq2024122408}
\begin{split}
\hat{{\bs \rho}}^{(k)}_{i} := &\, 2({\bs \epsilon}^{(k)}_{i} -{\bs \rho}^{(k)}_{i})\\
~  = &\, 2{\bs \sigma}^{\top}_{i}{\bs g}^{(k-1)}_{i-1} - 2{\bs \sigma}^{\top}_{i}{\bs g}^{(k-1)}_{k-1} -2\triangle f_{i} - {\bs \sigma}^{\top}_{i}{\bs G}^{(k-1)}{\bs \sigma}_{i}.
\end{split}
\end{equation}
As a result, \eqref{eq2024122405} can be simplified as
\begin{equation}\label{eq2024122409}
{\bs \sigma}^{\top}_{i}\triangle{\bs G}^{(k)}{\bs \sigma}_{i} = \hat{{\bs \rho}}^{(k)}_{i},  \quad\; k-n+1 \leq i \leq k.
\end{equation}

Next, we will construct a new objective function for updating $\triangle{\bs g}^{(k)}$ and $\triangle{\bs G}^{(k)}$ at the $k$-th iteration. Here it is worth pointing out that $(\triangle{\bs G}^{(k)})^{\top} = \triangle{\bs G}^{(k)}$ is assumed to consist with  $({\bs G}^{(k)})^{\top} = {\bs G}^{(k)}$, indicating that the required model points will be reduced, and the computational efficiency will be improved.
Motivated by the Least-Norm model given in \cite{powell04onupdating, xie2023derivativecombine, xieyuanH2p,xie2023dfoto},
we construct a new objective function as follow:
\begin{equation}\label{eq2024123001}
\Phi = \Phi_1 + \nu_k\Phi_2,
\end{equation}
where $\nu_k$ is a nonnegative constant, $\Phi_1:= \frac{1}{2}(\triangle {\bs g}^{(k)})^{\top}{\bs V} \triangle{\bs g}^{(k)}$, $\Phi_2 = \frac{1}{2}\|{\bs W}\triangle{\bs G}^{(k)}\|^2_{F}$ ($\|\cdot\|_{F}$ represents the Frobenius norm, i.e., when a matrix $\textbf{U} = (u_{ij})_{n\times n}$ is given, we define $\|\textbf{U}\|_{F} = (\sum_{i, j}u^2_{ij})^{1/2}$) and the matrices ${\bs V}$ and ${\bs W}$ are symmetric and positive definite. For any nonzero $\triangle{\bs g}^{(k)}$, it holds $\Phi_1, \Phi_2 > 0$.
the constant $\nu_k$ in \eqref{eq2024123001} is introduced to balance the magnitude of $\Phi_1$ and $\Phi_2$, and the reader is referred to \cite{powell04onupdating} for more details.\vspace{0.15cm}

Next, we will present a following theorem to illustrate a property of \eqref{eq2024123001}.\vspace{-0.05cm}
\begin{thm}\label{thm2025010101}
The objective function $\Phi$ defined in \eqref{eq2024123001} is a strictly convex function of $Q^{(k)}$.
\end{thm}
\begin{proof}
For the $k$-th iteration, we denote from \eqref{eq2024122301}
\begin{equation}\label{eq2025010101}
\nabla Q^{(k)} = {\bs g}^{(k)}_{k}  \quad\quad  \textrm{and}  \quad\quad  \nabla^2 Q^{(k)} = {\bs G}^{(k)}.
\end{equation}
With \eqref{eq2024122303}, it holds that
\begin{equation}\label{eq2025010102}
\triangle {\bs g}^{(k)} = \nabla Q^{(k)} - {\bs g}^{(k-1)}_{k-1}  \quad\quad   \textrm{and}  \quad\quad   \triangle{\bs G}^{(k)} = \nabla^2 Q^{(k)} - {\bs G}^{(k-1)},
\end{equation}
where ${\bs g}^{(k-1)}_{k-1}$ and ${\bs G}^{(k-1)}$ can be considered as the constants at the $k$-th iteration. Since ${\bs V}$ is symmetric positive definite, we can obtain
\begin{equation}\label{eq2025010103}
{\bs V} = {\bs V}_1{\bs V}^{\top}_1,
\end{equation}
where ${\bs V}_1$ is a lower triangular matrix with positive diagonal entries. Substituting \eqref{eq2025010102}-\eqref{eq2025010103} into \eqref{eq2024123001} has
\begin{equation}
\Phi = \|{\bs V}^{\top}_1(\nabla Q^{(k)}-{\bs g}^{(k-1)}_{k-1})\|^2_2
+ \nu_k\|{\bs W}(\nabla^2Q^{(k)} - {\bs G}^{(k-1)})\|^2_{F}.
\end{equation}
For any $\hat{Q}^{(k)}, \tilde{Q}^{(k)} \in \mathbb{P}_2$, we have
\begin{equation}\label{eq2025010201}
\begin{split}
~   & \Phi(\alpha\hat{Q}^{(k)} + (1-\alpha)\tilde{Q}^{(k)}) - (\alpha\Phi(\hat{Q}^{(k)}) + (1-\alpha)\Phi(\tilde{Q}^{(k)}))\\
=   &\,\, \|{\bs V}^{\top}_1(\alpha \nabla \hat{Q}^{(k)} + (1-\alpha)\nabla \tilde{Q}^{(k)} -{\bs g}^{(k-1)}_{k-1})\|^2_2\\
~  & +\nu_k\|{\bs W}( \alpha \nabla^2\hat{Q}^{(k)} + (1-\alpha)   \nabla^2\tilde{Q}^{(k)} - {\bs G}^{(k-1)})\|^2_{F}\\
~    & -\alpha(\|{\bs V}^{\top}_1(\nabla \hat{Q}^{(k)}-{\bs g}^{(k-1)}_{k-1})\|^2_2 + \nu_k\|{\bs W}(\nabla^2\hat{Q}^{(k)}-{\bs G}^{(k-1)})\|^2_{F})\\
~  & -(1-\alpha)(\|{\bs V}^{\top}_1(\nabla \tilde{Q}^{(k)}-{\bs g}^{(k-1)}_{k-1})\|^2_2 + \nu_k\|{\bs W}(\nabla^2\tilde{Q}^{(k)}-{\bs G}^{(k-1)})\|^2_{F})\\
=   &\, (\alpha^2 - \alpha)(\|{\bs V}^{\top}_1(\nabla\hat{Q}^{(k)}-\nabla\tilde{Q}^{(k)})\|^2_2 + \nu_k\|{\bs W}(\nabla^2\hat{Q}^{(k)} - \nabla^2\tilde{Q}^{(k)})\|^2_{F}).
\end{split}
\end{equation}
Since $\alpha\in (0, 1)$ and $\hat{Q}^{(k)}\neq\tilde{Q}^{(k)}$, from \eqref{eq2025010201} we can derive $\Phi < 0$, which means that $\Phi$ is a strictly convex function of $Q^{(k)}$.
\end{proof}
Combing with \eqref{eq2024122403} and \eqref{eq2024122409}, a constrained optimization problem is formed as follow:
\begin{equation}\label{eq2024122505}
\begin{split}
\min &\;\, \Phi\\
\textrm{s.t.} & \begin{cases}\vspace{0.2cm}
{\bs \sigma}^{\top}_{i}\triangle{\bs G}^{(k)}{\bs \sigma}_{i} = \hat{{\bs \rho}}^{(k)}_{i},\\
{\bs \sigma}^{\top}_{i}\triangle {\bs g}^{(k)} - {\bs \sigma}^{\top}_{i} \triangle{\bs G}^{(k)}({\bs \tau}_k - {\bs \tau}_i) = {\bs \epsilon}^{(k)}_{i},
\end{cases}
\end{split}
\end{equation}
where $k-n+1 \leq i \leq k$. Based on the \textbf{Theorem} \ref{thm2025010101} and the results given in \cite{powell04onupdating}, the constrained optimization problem \eqref{eq2024122505} has a unique solution. On the other hand, to solve \eqref{eq2024122505}, the Lagrange multiplier method is used, i.e.,\vspace{-0.1cm}
\begin{equation}
\begin{split}
\mathcal{L}(\triangle {\bs g}^{(k)}, \triangle {\bs G}^{(k)}, {\bs \eta}, {\bs \theta}) = & \;\Phi - \sum_{i=k-n+1}^{k}\eta_i(\frac{1}{2}{\bs \sigma}^{\top}_{i}\triangle{\bs G}^{(k)}{\bs \sigma}_{i} - \hat{{\bs \rho}}^{(k)}_{i})\\
&\; -\sum_{i=k-n+1}^{k}\theta_i({\bs \sigma}^{\top}_{i}\triangle {\bs g}^{(k)} - {\bs \sigma}^{\top}_{i} \triangle{\bs G}^{(k)}({\bs \tau}_k - {\bs \tau}_i) - {\bs \epsilon}^{(k)}_{i}),
\end{split}
\end{equation}
where ${\bs \eta} = (\eta_{k-n+1}, \cdots, \eta_{k})^{\top}$ and ${\bs \theta} = (\theta_{k-n+1}, \cdots, \theta_{k})^{\top}$. Based on the KKT conditions, we derive
\begin{equation}\label{eq2024122501}
\frac{\partial\mathcal{L}}{\partial\triangle{\bs g}^{(k)}} = {\bs V}\triangle{\bs g}^{(k)} - \sum_{i=k-n+1}^{k}{\bs \theta}_i{\bs \sigma}_i = {\bs 0},
\end{equation}
\begin{equation}\label{eq2024122502}
\frac{\partial\mathcal{L}}{\partial\triangle{\bs G}^{(k)}} = \nu_k{\bs W}^2\triangle{\bs G}^{(k)} - \frac{1}{2}\sum_{i=k-n+1}^{k}\eta_i{\bs \sigma}_i{\bs \sigma}^{\top}_i + \sum_{i=k-n+1}^{k}\theta_i{\bs \sigma}_i({\bs \tau}_k - {\bs \tau}_i)^{\top} = {\bs 0}.
\end{equation}
Let ${\bs \Lambda} := {\bs V}^{-1}$ and ${\bs M} := {\bs W}^{-2}$, and \eqref{eq2024122501}-\eqref{eq2024122502} become
\begin{equation}\label{eq2024122503}
\triangle{\bs g}^{(k)} = {\bs \Lambda}\sum_{i=k-n+1}^{k}{\bs \theta}_i{\bs \sigma}_i,
\end{equation}
\begin{equation}\label{eq2024122504}
\triangle{\bs G}^{(k)} = \frac{1}{2}{\bs M}(\sum_{i=k-n+1}^{k}\eta_i{\bs \sigma}_i{\bs \sigma}^{\top}_i - 2\sum_{i=k-n+1}^{k}\theta_i{\bs \sigma}_i({\bs \tau}_k - {\bs \tau}_i)^{\top}).
\end{equation}
Substituting \eqref{eq2024122503}-\eqref{eq2024122504} into the constrained conditions in \eqref{eq2024122505} is
\begin{equation}\label{eq2024122506}
{\bs \sigma}^{\top}_i{\bs M}(\sum_{j=k-n+1}^{k}\eta_i{\bs \sigma}_j{\bs \sigma}^{\top}_j - 2\sum_{j=k-n+1}^{k}\theta_j{\bs \sigma}_j({\bs \tau}_k - {\bs \tau}_j)^{\top}){\bs \sigma}_i = 2\nu_k\hat{\rho}^{(k)}_{i},
\end{equation}
\begin{equation}\label{eq2024122507}
2\nu_k{\bs \sigma}^{\top}_{i}{\bs \Lambda}\sum_{j=k-n+1}^{k}\theta_j{\bs \sigma}_j - {\bs \sigma}^{\top}_{i}{\bs M}(\sum_{j=k-n+1}^{k}\eta_i{\bs \sigma}_j{\bs \sigma}^{\top}_j - 2\sum_{j=k-n+1}^{k}\theta_j{\bs \sigma}_j({\bs \tau}_k - {\bs \tau}_j)^{\top})({\bs \tau}_k - {\bs \tau}_i) = 2\nu_k\epsilon^{(k)}_i.
\end{equation}
To show \eqref{eq2024122506}-\eqref{eq2024122507} more clearly, they are organized as follow:
\begin{equation}\label{eq2024122703}
\begin{bmatrix}\vspace{0.15cm}
{\bs H}_1  &-2{\bs B}\\
-{\bs H}_2 &2{\bs D}
\end{bmatrix}
\begin{bmatrix}\vspace{0.15cm}
{\bs \eta}\\
{\bs \theta}
\end{bmatrix}
=\begin{bmatrix}
2\nu_k\hat{{\bs \rho}}^{(k)}\\
2\nu_k{\bs \epsilon}^{(k)}
\end{bmatrix},
\end{equation}
where
\vspace{0.15cm}
\begin{subequations}
\begin{equation}
{\bs \eta} = (\eta_{k-n+1}, \cdots, \eta_{k})^{\top},  \quad\quad  {\bs \theta} = (\theta_{k-n+1}, \cdots, \theta_{k})^{\top},
\end{equation}
\vspace{0.05cm}
\begin{equation}
\hat{{\bs \rho}}^{(k)} = (\hat{\rho}^{(k)}_{k-n+1}, \cdots, \hat{\rho}^{(k)}_{k})^{\top},  \quad\quad   {\bs \epsilon}^{(k)} = (\epsilon^{(k)}_{k-n+1}, \cdots, \epsilon^{(k)}_{k})^{\top},
\end{equation}
\vspace{0.05cm}
\begin{equation}
{\bs A}_{ij} = ({\bs \sigma}^{\top}_{i+k-n}{\bs \sigma}_{j+k-n})(\bs \sigma_{i+k-n}^{\top} \bs M \bs \sigma_{j+k-n}),   \quad\;
{\bs B}_{ij} = {\bs \sigma}^{\top}_{i+k-n}{\bs M}{\bs \sigma}_{j+k-n}({\bs \tau}_k - {\bs \tau}_{j+k-n})^\top \bs \sigma_i,
\end{equation}
\vspace{0.05cm}
\begin{equation}
{\bs H}_1 = \textrm{diag}\{ \sum_{j=1}^{n}{\bs A}_{1j}, \cdots, \sum_{j=1}^{n}{\bs A}_{nj} \},  \quad\quad\;  {\bs H}_2 = \textrm{diag}\{ \sum_{j=1}^{n}{\bs B}_{1j}, \cdots, \sum_{j=1}^{n}{\bs B}_{nj} \},
\end{equation}
\vspace{0.05cm}
\begin{equation}
\hspace{0.25cm}{\bs D}_{ij} = {\nu_k \bs \sigma}^{\top}_{i+k-n}{\bs \Lambda}{\bs \sigma}_{j+k-n} + {\bs \sigma}^{\top}_{i+k-n}{\bs M}{\bs \sigma}_{j+k-n}({\bs \tau}_k - {\bs \tau}_{i+k-n})^{\top}({\bs \tau}_k - {\bs \tau}_{j+k-n}).\vspace{0.12cm}
\end{equation}
\end{subequations}
\eqref{eq2024122703} can be solved by the GMRES given in \cite{1986Youcef}. Then, $\triangle{\bs g}^{(k)}$ and $\triangle{\bs G}^{(k)}$ are obtained by substituting ${\bs \eta}$ and ${\bs \theta}$ into \eqref{eq2024122503}-\eqref{eq2024122504}.

To improve the computational efficiency, we will relax the constrained conditions shown in \eqref{eq2024122505}. In addition, to distinguish $\triangle{\bs g}^{(k)}$ and $\triangle{\bs G}^{(k)}$ presented in \eqref{eq2024122505}, we use $\triangle{\bs g}_{*}^{(k)}$ and $\triangle{\bs G}_{*}^{(k)}$ to replace them. As a result,
a new constrained optimization problem is formed as follow:
\begin{equation}\label{eq20241227001}
\begin{split}
\min &\;\, \Phi := \|{\bs V}^{\top}_1\triangle {\bs g}_{*}^{(k)}\|_2^2 + \nu_k\|{\bs W}\triangle {\bs G}_{*}^{(k)}\|_{F}^2\\
\textrm{s.t.} & \begin{cases}\vspace{0.2cm}
\frac{1}{2}{\bs \sigma}^{\top}_{k}\triangle{\bs G}_{*}^{(k)}{\bs \sigma}_{k} = \check{{\bs \rho}}^{(k)}_{k},\\
{\bs \sigma}^{\top}_{k}\triangle {\bs g}_{*}^{(k)} = \hat{{\bs \epsilon}}^{(k)}_{k},
\end{cases}
\end{split}
\end{equation}
where
\begin{equation}\label{eq20241227002}
\check{{\bs \rho}}^{(k)}_{k} = -\triangle f_{k} - \frac{1}{2}{\bs \sigma}^{\top}_{k}{\bs G}^{(k-1)}{\bs \sigma}_{k}   \quad\quad  \textrm{and} \quad\quad
\hat{{\bs \epsilon}}^{(k)}_{k} = - {\bs \sigma}^{\top}_{k}{\bs g}^{(k-1)}_{k-1}.
\end{equation}
Based on \eqref{eq2024122503}-\eqref{eq2024122504}, we can derive that
\begin{equation}\label{eq20241227004}
\triangle {\bs g}_{*}^{(k)} = \theta_k{\bs \Lambda}{\bs \sigma}_k
\quad\quad  \textrm{and} \quad\quad
\triangle {\bs G}_{*}^{(k)} = \frac{1}{2\nu_k}\eta_k{\bs M}{\bs \sigma}_{k}{\bs \sigma}^{\top}_{k}.
\end{equation}
Substituting \eqref{eq20241227004} into \eqref{eq20241227001} holds that
\begin{equation}\label{eq2024122801}
\eta_k = -\frac{4\triangle f_{k} + 2{\bs \sigma}^{\top}_{k}{\bs G}^{(k-1)}{\bs \sigma}_k}{{\bs \sigma}^{\top}_{k}{\bs \sigma}_{k}{\bs \sigma}^{\top}_{k}{\bs M}{\bs \sigma}_k}\nu   \quad\quad   \textrm{and}  \quad\quad   \theta_{k} = -\frac{{\bs \sigma}^{\top}_{k}{\bs g}^{(k-1)}_{k-1}}{{\bs \sigma}^{\top}_{k}{\bs \Lambda}{\bs \sigma}_{k}}.
\end{equation}
With \eqref{eq2024122801}, we can obtain from \eqref{eq20241227004}
\begin{equation}\label{eq2024122802}
\triangle {\bs g}_{*}^{(k)} = -\frac{{\bs \sigma}^{\top}_k{\bs g}^{(k-1)}_{k-1}}{{\bs \sigma}^{\top}_k{\bs \Lambda}{\bs \sigma}_k}{\bs \Lambda}{\bs \sigma}_k  \quad\quad \textrm{and}  \quad\quad \triangle {\bs G}_{*}^{(k)} = -\frac{2\triangle f_{k} + {\bs \sigma}^{\top}_{k}{\bs G}^{(k-1)}{\bs \sigma}_k}{\nu_k{\bs \sigma}^{\top}_{k}{\bs \sigma}_{k}{\bs \sigma}^{\top}_{k}{\bs M}{\bs \sigma}_k}{\bs M}{\bs \sigma}_k{\bs \sigma}^{\top}_k.
\end{equation}
For simplicity, in the current paper we consider ${\bs \Lambda} = I$ and ${\bs M} = I$, where $I\in \mathbb{R}^{n\times n}$ is the identity matrix. As a result, we have
\begin{equation}\label{eq2024122803}
{\bs \sigma}^{\top}_{k}{\bs M}{\bs \sigma}_k = \|{\bs \sigma}_k\|_2^2   \quad\quad  \textrm{and}  \quad\quad  {\bs \sigma}^{\top}_{k}{\bs \Lambda}{\bs \sigma}_k = \|{\bs \sigma}_k\|_2^2.
\end{equation}
\eqref{eq2024122801}-\eqref{eq2024122802} become
\begin{subequations}
\begin{equation}\label{eq2025012001}
\eta_k = -\frac{4\triangle f_{k} + 2{\bs \sigma}^{\top}_{k}{\bs G}^{(k-1)}{\bs \sigma}_k}{\|{\bs \sigma}_k\|^4_2}   \quad\,   \textrm{and}  \quad\,   \theta_{k} = -\frac{{\bs \sigma}^{\top}_{k}{\bs g}^{(k-1)}_{k-1}}{\|{\bs \sigma}_{k}\|^2_2},
\end{equation}
\begin{equation}\label{eq2025010202}
\hspace{2cm}\triangle {\bs g}_{*}^{(k)} = -\frac{{\bs \sigma}^{\top}_k{\bs g}^{(k-1)}_{k-1}{\bs \sigma}_k}{\|{\bs \sigma}_k\|^2_2}  \quad\, \textrm{and}  \quad\, \triangle {\bs G}_{*}^{(k)} = -\frac{(2\triangle f_{k} + {\bs \sigma}^{\top}_{k}{\bs G}^{(k-1)}{\bs \sigma}_k){\bs \sigma}_k{\bs \sigma}^{\top}_k}{\nu_k\|{\bs \sigma}_k\|^4_2}.
\end{equation}
\end{subequations}
From \eqref{eq2025010202}, we can derive that

\begin{subequations}
\begin{equation}\label{eq2025010203}
\hspace{2cm}{\bs \sigma}^{\top}_{k}{\bs g}^{(k-1)}_{k-1} = 0   \quad\quad   \Rightarrow  \quad\quad  \triangle{\bs g}^{(k)}_{*} = {\bs 0}, \vspace{0.1cm}
\end{equation}
\begin{equation}\label{eq2025010204}
\hspace{2cm}\hspace{-1.7cm}2\triangle f_{k} + {\bs \sigma}^{\top}_{k}{\bs G}^{(k-1)}{\bs \sigma}_k = 0   \quad\quad   \Rightarrow  \quad\quad  \triangle{\bs G}^{(k)}_{*} = {\bs 0}.\vspace{0.15cm}
\end{equation}
\end{subequations}
In other words, when ${\bs \sigma}^{\top}_{k}{\bs g}^{(k-1)}_{k-1} = 0$, the gradient at the previous iteration can be used to replace the gradient at the $k$-th iteration. When $\triangle f_{k} + {\bs \sigma}^{\top}_{k}{\bs G}^{(k-1)}{\bs \sigma}_k = 0$, the Hessian matrix at the $k$-th iteration can be replaced by the Hessian matrix at the previous iteration.

Next, based on \eqref{eq20241227001}, we focus on a bounded analysis for constrained conditions shown in \eqref{eq2024122505} with $k-n+1\leq i \leq k-1$.
Let
\begin{equation}\label{eq2025012201}
\mathcal{E}^{(1)}_i := {\bs \sigma}^{\top}_{i}\triangle{\bs G}^{(k)}{\bs \sigma}_{i} - \hat{{\bs \rho}}^{(k)}_{i},
\end{equation}
\begin{equation}\label{eq2025012202}
\mathcal{E}^{(2)}_i := {\bs \sigma}^{\top}_{i}\triangle {\bs g}^{(k)} - {\bs \sigma}^{\top}_{i} \triangle{\bs G}^{(k)}({\bs \tau}_k - {\bs \tau}_i) - {\bs \epsilon}^{(k)}_{i}.
\end{equation}
Substituting \eqref{eq2024122702}, \eqref{eq2025010202}, \eqref{eq2025010202} and \eqref{eq2024122408} into \eqref{eq2025012201} becomes
\begin{equation}\label{eq2025012203}
\begin{split}
\mathcal{E}^{(1)}_{i} & = \frac{(2\triangle f_{k} + {\bs \sigma}^{\top}_{k}{\bs G}^{(k-1)}{\bs \sigma}_k)({\bs \sigma}^{\top}_i{\bs \sigma}_k)^2}{\|{\bs \sigma}_k\|^4_2} - 2{\bs \sigma}^{\top}_{i}({\bs g}^{(k-1)}_{i-1} - {\bs g}^{(k-1)}_{k-1}) + 2\triangle f_{i} + {\bs \sigma}^{\top}_{i}{\bs G}^{(k-1)}{\bs \sigma}_{i}\\
~   & = -\frac{\nu_k\eta_k\|{\bs \sigma}_k\|^4_2({\bs \sigma}^{\top}_i{\bs \sigma}_k)^2}{2\|{\bs \sigma}_k\|^4_2} - 2{\bs \sigma}^{\top}_{i}({\bs g}^{(k-1)}_{i-1} - {\bs g}^{(k-1)}_{k-1}) + 2{\bs \sigma}^{\top}_{i}{\bs g}^{(k)}_{i-1} + 2{\bs \sigma}^{\top}_{i}{\bs G}^{(k)}{\bs \sigma}_{i}\\
~ &= -\frac{1}{2}\nu_k\eta_k({\bs \sigma}^{\top}_i{\bs \sigma}_k)^2 + 2{\bs \sigma}^{\top}_{i}{\bs g}^{(k-1)}_{k-1}+2{\bs \sigma}^{\top}_{i}{\bs G}^{(k)}{\bs \sigma}_{i}.
\end{split}
\end{equation}
On the other hand, with \eqref{eq2024122302}, \eqref{eq2024122407}, \eqref{eq2025012001} and \eqref{eq2025010202}, we can derive that
\begin{equation}\label{eq2025012204}
\begin{split}
\mathcal{E}^{(2)}_{i} = & -{\bs \sigma}^{\top}_i\frac{{\bs \sigma}^{\top}_k{\bs g}^{(k-1)}_{k-1}{\bs \sigma}_k}{\|{\bs \sigma}_k\|^2_2} + {\bs \sigma}^{\top}_{i} \frac{(2\triangle f_{k} + {\bs \sigma}^{\top}_{k}{\bs G}^{(k-1)}{\bs \sigma}_k){\bs \sigma}_k{\bs \sigma}^{\top}_k}{\nu_k\|{\bs \sigma}_k\|^4_2}({\bs \tau}_k - {\bs \tau}_i) \\
~   &\, - ({\bs \sigma}^{\top}_{i}{\bs G}^{(k-1)}({\bs \tau}_k - {\bs \tau}_i) - {\bs \sigma}^{\top}_{i}{\bs g}^{(k-1)}_{k-1})\\
~  =&\; {\bs \sigma}^{\top}_i({\bs g}^{(k-1)}_{k-1} - \frac{{\bs \sigma}^{\top}_k{\bs g}^{(k-1)}_{k-1}{\bs \sigma}_k}{\|{\bs \sigma}_k\|^2_2}) + {\bs \sigma}^{\top}_i(\frac{\eta_k\|{\bs \sigma}_k\|^4_2{\bs \sigma}_k{\bs \sigma}^{\top}_k}{2\nu_k\|{\bs \sigma}_k\|^4_2} - {\bs G}^{(k-1)})({\bs \tau}_k - {\bs \tau}_i)\\
~  = &\; {\bs \sigma}^{\top}_i({\bs g}^{(k-1)}_{k-1} - \frac{{\bs \sigma}^{\top}_k{\bs g}^{(k-1)}_{k-1}{\bs \sigma}_k}{\|{\bs \sigma}_k\|^2_2} + (\frac{\eta_k{\bs \sigma}_k{\bs \sigma}^{\top}_k}{2\nu_k} - {\bs G}^{(k-1)})\sum_{j=i+1}^{k}{\bs \sigma}_j).
\end{split}
\end{equation}
When ${\bs g}^{(k)}$ and ${\bs G}^{(k)}$ is bounded for each iteration, from \eqref{eq2025012203} and \eqref{eq2025012204} we can obtain
\begin{equation}\label{eq2025012205}
|\mathcal{E}^{(r)}_i| = \mathcal{O}(\|{\bs \sigma}_i\|_2),  \quad\,    r = 1, 2.
\end{equation}
From \eqref{eq2025012205}, we can conclude that when $\|{\bs \sigma}_i\| \rightarrow 0$ (i.e., the iteration is convergent), the constrained conditions $|\mathcal{E}^{(r)}_i| $ will tend to zero. In other words, the simplified constrained  problem \eqref{eq20241227001} will be equivalent to the original constrained problem \eqref{eq2024122505} with $\|{\bs \sigma}_i\| \rightarrow 0$, which means that our proposed model \eqref{eq20241227001} is feasible for solving \eqref{eq2024121301}. Moreover, based on the proposed model \eqref{eq20241227001}, ${\bs g}^{(k)}$ and ${\bs G}^{(k)}$ used in the iteration can be explicitly updated by using \eqref{eq2025010202}, indicating that the computational efficiency will be improved greatly.\vspace{0.15cm}

Finally, since ${\bs g}^{(k)}$ and ${\bs G}^{(k)}$ is used to update the model points, it is very necessary to estimate $\|{\bs g}^{(k)}_{i} - \nabla f({\bs x})\|_2$ and $\| {\bs G}^{(k)} - \nabla^2f({\bs x}) \|_2$. Therefore, we give a following theorem to show approximation results on them.

\begin{thm}\label{thm2025012801}
It is assumed that $\nabla^2 f({\bs x})$ is Lipschitz continuous with a positive constant $L_{f}$. Let $\mathcal{B}(\hat{{\bs x}}, \rho)$ be a ball including model points $\{ {\bs x}_i \}_{i = k-n+1}^k$, where $\hat{{\bs x}}$ represents the center of the ball, and $\rho$ represents the radius of the ball. For the model \eqref{eq2024122505}, we have
\begin{equation}\label{eq2025020903}
\begin{cases}\vspace{0.2cm}
\|{\bs g}^{(k)}_{i} - \nabla f({\bs x})\|_2 \leq 2 \mathcal{L}_1\delta^2_1    \\
\| {\bs G}^{(k)} - \nabla^2f({\bs x}) \|_2 \leq \mathcal{L}_1\delta_1,
\end{cases}
\quad\;   \forall {\bs x} \in \mathcal{B}(\hat{{\bs x}}, \rho).
\end{equation}
On the other hand, for the model \eqref{eq20241227001}, we have
\begin{equation}\label{eq2025020904}
\begin{cases}\vspace{0.2cm}
\|{\bs g}^{(k)}_{k} - \nabla f({\bs x})\|_2 \leq 2 \mathcal{L}_2\delta^2_2    \\
\| {\bs G}^{(k)} - \nabla^2f({\bs x}) \|_2 \leq \mathcal{L}_2\delta_2,
\end{cases}
\quad\;  \forall {\bs x} \in \mathcal{B}(\hat{{\bs x}}, \rho).
\end{equation}
In \eqref{eq2025020903}-\eqref{eq2025020904},
$\mathcal{L}_1$ and $\mathcal{L}_2$ are positive constants depending on $L_f,n$ and $L_f$ respectively, and
\begin{equation}
\delta_1 = \max\limits_{k-n+1 \leq i \leq k}\|{\bs x}_i - {\bs x}\|_2,   \quad\quad   \delta_2 = \max\{ \|{\bs x}_{k-1} - {\bs x}\|_2, \|{\bs x}_k - {\bs x}\|_2 \}.
\end{equation}
\end{thm}

\begin{proof}
Let $Q^{(k)}_{i} := Q^{(k)}({\bs x}_i)$, and we have
\begin{equation}\label{DFO_formu34}
Q^{(k)}_{i}-Q(\bs{x})=f_{i}-f(\bs{x})+\kappa^{f}(\bs{x}), \quad \quad  k-n+1\leq i \leq k,
\end{equation}
where $\kappa^{f}(\bs{x}) = f(\bs{x})-Q(\bs{x})$. Based on the Taylor Series Formula, we obtain
\begin{equation}\label{DFO_formu25}
f_{i}-f(\bs{x}) =\nabla f(\bs{x})^{\top} (\bs{x}_{i}-\bs{x})+\frac{1}{2}(\bs{x}_{i}-\bs{x})^{\top}  \nabla^2 f(\bs{x})(\bs{x}_{i}-\bs{x}) + \mathcal{O}\left(\|\bs{x}_{i}-\bs{x}\|^3_2\right).
\end{equation}
Based on the Lemma 4.14 shown in \cite{dennisnumerical1996} and the Lipschitz condition of $\nabla^2 f(\bs{x})$,
we get
\begin{equation}\label{eq2025021002}
\begin{aligned}
\mathcal{O}\left(\|\bs{x}_{i}-\bs{x}\|^3_2\right)
 \leq L_{f} \delta_1^3.
\end{aligned}
\end{equation}
Meanwhile, with \eqref{eq2025020902}, we can obtain
\begin{equation}\label{eq2025020901}
\begin{aligned}
Q^{(k)}_{i}-Q^{(k)}_{i-1}=\left(f_{i}-f(\bs{x})\right)-\left(f_{i-1}-f(\bs{x})\right).
\end{aligned}
\end{equation}
Based on \eqref{eq2024122702}, \eqref{DFO_formu25} and \eqref{eq2025020901}, it can be derived that
\begin{equation}\label{eq2025020905}
\begin{aligned}
&
{\bs \sigma}^{\top}_{i}\left({\bs g}^{(k)}_{i} -
\nabla f(\bs{x}) -\left(\nabla^2 f(\bs{x})-\bs{G}^{(k)}\right)(\bs{x}_{i-1}-\bs{x}) \right)+ \frac{1}{2}{\bs \sigma}^{\top}_{i}\left({\bs G}^{(k)}-\nabla^2 f(\bs{x})\right){\bs \sigma}_{i}\\
&+{\bs \sigma}^{\top}_{i} \left({\bs g}^{(k)}_{i-1}-{\bs g}^{(k)}_{i}-{\bs G}^{(k)} (\bs{x}_{i-1}-\bs{x})\right)=\mathcal{O}\left(\delta_1^3\right), \quad\;  k-n+1 \leq i \leq k.
\end{aligned}
\end{equation}
From \eqref{eq2025021002}, the right-hand side of \eqref{eq2025020905} is bounded by $2L_{f} \delta_1^3$. To facilitate the subsequent analysis, \eqref{eq2025020905} is reformulated into a following form
\begin{equation}\label{eq2025020906}
{\bs C}{\bs Y} = \mathcal{O}(\delta_1^3){\bs 1}_{n},
\end{equation}
where
\begin{equation}\label{DFO_formu15}
\begin{aligned}
& {\bs 1}_n = (1, \cdots, 1)^{\top},  \quad\;
\bm{C}_{1}=\text{diag}\left({\bm \sigma}_{k-n+1}^{\top}~~{\bm\sigma}_{k-n+2}^{\top}~~\cdots ~~{\bm \sigma}_{k}^{\top} \right),\\
& \bm{C}_{2}=\left[({\bm \sigma}_{k-n+1}\otimes{\bm \sigma}_{k-n+1})^{\top}~~({\bm \sigma}_{k-n+2}\otimes{\bm \sigma}_{k-n+2})^{\top}~~\cdots ~~({\bm \sigma}_{k}\otimes{\bm \sigma}_{k})^{\top}\right]^{\top},\\
&\bs{s}_{i}(\bs{x})={\bs g}^{(k)}_{i-1}-{\bs g}^{(k)}_{i}-{\bs G}^{(k)} (\bs{x}_{i-1}-\bs{x}),~~ \;\; h_{jl}(\bs{x})= \left({\bs G}^{(k)}-\nabla^2 f(\bs{x})\right)_{jl},  \quad\; 1\leq j , l \leq n,
\\
&\bs{u}_{i}(\bs{x}) = {\bs g}^{(k)}_{i} -
\nabla f(\bs{x}) - \left(\nabla^2 f(\bs{x})-\bs{G}^{(k)}\right)(\bs{x}_{i-1}-\bs{x}), \\
&\bs{Y}=[ \bs{Y}_{1}~~\bs{Y}_{2} ~~\bs{Y}_{3} ]^{\top}, \quad\;
\bs{Y}_{1}= [\bs{s}_{k-n+1}(\bs{x})~~\bs{s}_{k-n+2}(\bs{x})~~\cdots~~\bs{s}_{k}(\bs{x})]^{\top} , \\
&\bs{Y}_{2}= [\bs{u}_{k-n+1}(\bs{x})~~\bs{u}_{k-n+2}(\bs{x})~~\cdots~~\bs{u}_{k}(\bs{x})]^{\top} ,\\
&\bs{Y}_{3} = [h_{11}(\bs{x})~~\cdots~~h_{1n}(\bs{x})~~h_{21}(\bs{x})~~\cdots~~ h_{nn}(\bs{x})]^{\top}.\\
\end{aligned}
\end{equation}
To remove the dependence of ${\bs C}$ on $\delta_1$, we introduce the scaled matrix
\begin{equation}\label{DFO_formu20}
\bm{C}_{*}=\bm{C}
\begin{bmatrix}\vspace{0.15cm}
{\bm{D}^{-1}_{1}}  & \bm{0} & \bm{0}\\
\bm{0} & {\bm{D}^{-1}_{1}}  &  \bm{0}\\
\bm{0} & \bm{0}  &{\bm{D}^{-1}_{2}}
\end{bmatrix},
\end{equation}
where
$\bm{D}_{1}= \delta_1 \bs{I}_{n^2}$, $\bm{D}_{2}=  \delta_1^2 \bs{I}_{n^2}$,  $\bs{I}_{n^2}$ is a $n^2\times n^2$ identity matrix. It is worth pointing out that all elements of $\bm{C}_{*}$ are less than 1, and independent of $\delta_1$. As a result, the left-hand side of \eqref{eq2025020906} is equivalent to
\begin{equation}\label{DFO_formu3}
\bm{C}_{*}
\begin{bmatrix}\vspace{0.15cm}
{\bm{D}_{1}}  & \bm{0} & \bm{0}\\
\bm{0} &  {\bm{D}_{1}}   & \bm{0}  \\
\bm{0}  & \bm{0}  &{\bm{D}_{2}}
\end{bmatrix}
\begin{bmatrix}\vspace{0.15cm}
\bs{Y}_{1}\\
\bs{Y}_{2}\\
\bs{Y}_{3}
\end{bmatrix}
=\bm{C}_{*}
\begin{bmatrix}\vspace{0.15cm}
\delta_1 \bs{I}_{n^2 }\bs{Y}_{1}\\
\delta_1 \bs{I}_{n^2 }\bs{Y}_{2}\\
\delta^2_1 \bs{I}_{n^2 } \bs{Y}_{3}
\end{bmatrix}.
\end{equation}
Based on \eqref{eq2025021002}, \eqref{eq2025020906} and \eqref{DFO_formu3}, we can obtain
\begin{equation}\label{DFO_formu4}
\left\|
\begin{bmatrix}\vspace{0.15cm}
\delta_1 \bs{I}_{n^2 }\bs{Y}_{1}\\
\delta_1 \bs{I}_{n^2 }\bs{Y}_{2}\\
\delta^2_1 \bs{I}_{n^2 } \bs{Y}_{3}
\end{bmatrix}\right\|_{2} \leq 2 n^{\frac{1}{2}}L_{f} \delta_1^3\|\bm{C}^{-1}_{*}\|_2,
\end{equation}
where $\|\bm{C}^{-1}_{*}\|_2$ denotes the $2$-norm of a pseudo-inverse matrix $\bm{C}_{*}$, and $\|\bm{C}^{-1}_{*}\|_2=\frac{1}{\lambda_{\min}(\bm{C}_{*})}$, ($\lambda_{\min}(\bm{C}_{*})$ denotes the smallest nonzero singular value of $\bm{C}_{*}$, the reader is referred to Chapter 14 in \cite{Gallier2011} for more details). From \eqref{DFO_formu4}, we further derive that
\begin{equation}\label{eq2025020907}
\begin{cases}\vspace{0.15cm}
\left\|  \bs{Y}_{2}\right\|_2 \leq  2 n^{\frac{1}{2}}L_{f} \delta^2_1\|\bm{C}^{-1}_{*}\|_2,
\\
\left\|  \bs{Y}_{3}\right\|_2 \leq 2 n^{\frac{1}{2}}L_{f} \delta_1 \|\bm{C}^{-1}_{*}\|_2.
\end{cases}
\end{equation}
Based on the relationship between the norms of matrices $\|\cdot\|_2 \leq \|\cdot\|_{F}$, from \eqref{eq2025020907} we can obtain
\begin{equation}\label{eq2025021001}
\begin{aligned}
&\left\|{\bs G}^{(k)}-\nabla^2 f(\bs{x}) \right\|_2  \leq \left\| {\bs G}^{(k)}-\nabla^2 f(\bs{x})\right\|_{F}= \left\|  \bs{Y}_{3} \right\|_2 \leq 2n^{\frac{1}{2}}L_{f} \delta_1 \|\bm{C}^{-1}_{*}\|_2.
\end{aligned}
\end{equation}
Similarly, from \eqref{eq2025020907} the following result can also be obtained
\begin{equation}\label{DFO_formu28}
\begin{aligned}
&\left\| {\bs g}^{(k)}_{i} -
\nabla f(\bs{x})\right\|_2  \leq  \left\|  \bs{Y}_{2}\right\|_2
+ \left\|{\bs G}^{(k)}-\nabla^2 f(\bs{x}) \right\|_2  \delta_1 \leq 4 n^{\frac{1}{2}}L_{f} \delta^2_1\|\bm{C}^{-1}_{*}\|_2.
\end{aligned}
\end{equation}
Let $\mathcal{L}_1 := 2 n^{\frac{1}{2}}L_{f} \|\bm{C}^{-1}_{*}\|_2$, and \eqref{eq2025020903} is proved with \eqref{eq2025021001}-\eqref{DFO_formu28}.

Next, we will consider to prove \eqref{eq2025020904}.
Since $Q^{(k)}_{k}=f_{k}$ and $Q^{(k-1)}_{k-1}=f_{k-1}$, we have
\begin{equation}\label{DFO_formu47}
\begin{aligned}
Q^{(k)}_{k}-Q^{(k-1)}_{k-1}= & \left(\bs{g}^{(k)}_{k}-\bs {g}^{(k-1)}_{k-1}\right)^{\top} \bs{x}_{k-1} + \left(\bs {g}^{(k)}_{k}\right)^{\top} \left(\bs{x}_{k}-\bs{x}_{k-1}\right)
-\left(\bs{x}_{k}-\bs{x}_{k-1}\right)^{\top} \bs{G}^{(k)} \bs{x}_{k-1}\\
& \,
-\frac{1}{2} \left(\bs{x}_{k}-\bs{x}_{k-1}\right)^{\top} \bs{G}^{(k)} \left(\bs{x}_{k}-\bs{x}_{k-1}\right)
-\frac{1}{2} \bs{x}_{k-1} \left(\bs{G}^{(k)}-\bs{G}^{(k-1)} \right)\bs{x}_{k-1}  .
\end{aligned}
\end{equation}
On the other hand, similar to \eqref{eq2025020905}, we can derive
\begin{equation}\label{DFO_formu47}
\begin{aligned}
&\left(\bs{g}^{(k)}_{k}-\bs {g}^{(k-1)}_{k-1}\right)^{\top} \bs{x}_{k-1}-\frac{1}{2} \bs{x}_{k-1}^{\top} \left(\bs{G}^{(k)}-\bs{G}^{(k-1)} \right) \bs{x}_{k-1}
- \left(\bs{x}_{k}-\bs{x}_{k-1}\right)^{\top} \bs{G}^{(k)}\bs{x}\\
&
+ \left(\bs{x}_{k}-\bs{x}_{k-1}\right)^{\top}  \nabla^2 f(\bs{x})\left(\bs{x}_{k} +\bs{x}_{k-1} -2\bs{x}\right)
-\frac{1}{2} \left(\bs{x}_{k}-\bs{x}_{k-1}\right)^{\top} \left( \bs{G}^{(k)} -\nabla^2 f(\bs{x})  \right)\left(\bs{x}_{k}-\bs{x}_{k-1}\right) \\
&+
\left(\bs{x}_{k}-\bs{x}_{k-1}\right)^{\top}\left( \bs {g}^{(k)}_{k}-  \nabla f(\bs{x}) -\left(\bs{G}^{(k)}-\nabla^2 f(\bs{x})\right)(\bs{x}_{k-1}-\bs{x}) \right) =\mathcal{O}\left(\delta_2^3\right).
\end{aligned}
\end{equation}
Similar to \eqref{eq2025021002}, the upper bound of the right-hand side of \eqref{DFO_formu47} is described by
\begin{equation}\label{eq2025021003}
\begin{aligned}
\mathcal{O}\left(\|\bs{x}_{k}-\bs{x}\|^3_2\right)
+\mathcal{O}\left(\|\bs{x}_{k-1}-\bs{x}\|^3_2\right) \leq L_{f} \delta_2^3 + L_{f} \delta_2^3=2L_{f} \delta_2^3.
\end{aligned}
\end{equation}
To write \eqref{DFO_formu47} as a simpler form, some notations should be introduced as follows:
\begin{equation}\label{DFO_formu49}
\bm{M}  :=
\begin{bmatrix}\vspace{0.15cm}
\bm{M}_{1} &  {0}   &  \bm{0}   &  \bm{0} \\
{0}  & \bm{M}_{2}  &  \bm{0}   &  \bm{0}\\
 \bm{0}   &  \bm{0}   &   \bm{M}^{\top}_{3}  & \bm{0} \\
  \bm{0}   &  \bm{0}   & \bm{0}  &   \bm{M}^{\top}_{4}   \\
\end{bmatrix}
\begin{bmatrix}\vspace{0.03cm}
1 \\
1 \\
1 \\
1
\end{bmatrix}, \quad\; \bm{M}_{1}=    \left(\bs{g}^{(k)}_{k}-\bs {g}^{(k-1)}_{k-1}\right)^{\top} \bs{x}_{k-1},\\
\end{equation}
\begin{equation}\label{DFO_formu50}
\begin{aligned}
&\bm{M}_{2}=    -\frac{1}{2} \bs{x}_{k-1}^{\top} \left(\bs{G}^{(k)}-\bs{G}^{(k-1)} \right)\bs{x}_{k-1}- \left(\bs{x}_{k}-\bs{x}_{k-1}\right)^{\top} \left(\bs{G}^{(k)} \bs{x}-  \nabla^2 f(\bs{x})\left(\bs{x}_{k} +\bs{x}_{k-1} -2\bs{x}\right)\right),\\
&\bm{M}_{3}= \bs{x}_{k}-\bs{x}_{k-1} ,
\quad\; \bm{M}_{4}=\left(\bs{x}_{k}-\bs{x}_{k-1}\right)\otimes \left(\bs{x}_{k}-\bs{x}_{k-1}\right),
    \end{aligned}
\end{equation}
\begin{equation}\label{DFO_formu51}
\begin{aligned}
&\bs{V}=[ 1~~1~~\bs{V}_{1} ~~\bs{V}_{2} ]^{\top}, \quad\; \bs{V}_{1}= \left[{\bs g}^{(k)}_{k} -
\nabla f(\bs{x}) - \left(\bs{G}^{(k)}-\nabla^2 f(\bs{x})\right)(\bs{x}_{k-1}-\bs{x})\right]^{\top},\\
&\bs{V}_{2} = \left[h_{11}(\bs{x})~~\cdots~~h_{1n}(\bs{x})~~h_{21}(\bs{x})~~\cdots~~ h_{nn}(\bs{x})\right]^{\top} , \quad\;  h_{jl}(\bs{x})= \left({\bs G}^{(k)}-\nabla^2 f(\bs{x})\right)_{jl} \quad 1\leq j , l \leq n.
\end{aligned}
\end{equation}
Then, \eqref{DFO_formu47} is formed as
\begin{equation}\label{DFO_formu56}
\bm{M}^{\top} \bs{V}=\mathcal{O}\left(\delta_2^3\right).
\end{equation}
A scale matrix is introduced as follow:
\begin{equation}\label{DFO_formu52}
\bm{M}_{*}=
\begin{bmatrix}\vspace{0.15cm}
1 & 0 & \bm{0} & \bm{0}\\
0 & 1 & \bm{0} & \bm{0}\\
\bm{0}& \bm{0} & {\bm{R}^{-1}_{1}}  & \bm{0}\\
\bm{0}  & \bm{0} & \bm{0} &{\bm{R}^{-1}_{2}}
\end{bmatrix}\bm{M},
\end{equation}
where $\bm{R}_{1}= \delta_2\bs{I}_{n}$ and $\bm{R}_{2}=  \delta_2^2 \bs{I}_{n}$.
Similar to the derivations for \eqref{DFO_formu3}-\eqref{eq2025020907}, we can obtain
\begin{subequations}
\begin{equation}\label{DFO_formu68}
\left\| {\bs G}^{(k)}-\nabla^2 f(\bs{x})\right\|_2  \leq \left\| {\bs G}^{(k)}-\nabla^2 f(\bs{x})\right\|_{F}=\left\|  \bs{V}_{2}\right\|_2
 \leq  2 L_{f}  \delta_2 \|(\bm{M}^{\top}_{*})^{-1}\|_2,
\end{equation}
\begin{equation}\label{DFO_formu60}
\begin{aligned}
\left\| {\bs g}^{(k)}_{k} -
\nabla f(\bs{x})\right\|_2
\leq \left\|  \bs{V}_{1}\right\|_2
+ \left\| \bs{V}_{2} \right\|_2  \delta_2 \leq  4 L_{f} \delta_2^2\|(\bm{M}^{\top}_{*})^{-1}\|_2 .
\end{aligned}
\end{equation}
\end{subequations}
Let $\mathcal{L}_2 := 2L_{f}\|(\bm{M}^{\top}_{*})^{-1}\|_2$, \eqref{eq2025020904} is proved with \eqref{DFO_formu68}-\eqref{DFO_formu60}.

\end{proof}

\vspace{-0.35cm}
\section{A novel algorithm for computing \eqref{eq2024121301}}\label{sect3}

In this section, based on the results presented in section \ref{sect2}, we focus on proposing a novel algorithm for computing \eqref{eq2024121301}. As shown in Fig.\ref{figure2025011301}, the main configuration process is divided into three steps.

\begin{figure}[h]
\centering
\includegraphics[width=8cm]{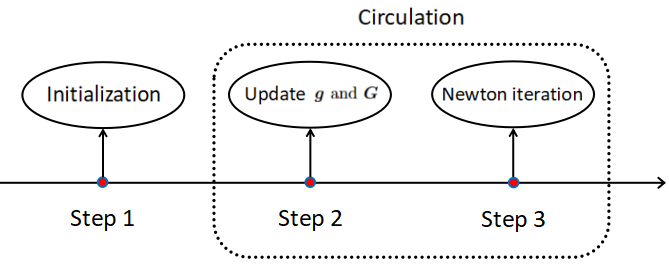}\\\vspace{-0.2cm}
\caption{A configuration of our proposed algorithm for computing \eqref{eq2024121301}.}\label{figure2025011301}
\end{figure}

\begin{center}
\begin{minipage}{0.8\linewidth}
\begin{breakablealgorithm}
\setstretch{1.1}
\caption{A novel iterative algorithm for computing \eqref{eq2024121301}}\label{algorithm2025011301}
\begin{algorithmic}[1]
\Require Given $\epsilon$
\Ensure The solution, denoted by ${\bs x}^{*}$
\State Randomly generate $n$ model points, denoted by $\{{\bs x}_i\}_{i=0}^{n-1}$;    \quad\; {\color{blue}$\boldsymbol{\triangleright}$} Initialization
\State \textbf{For} $k = n-1,~ n,~ n+1,~ \cdots$    \quad\quad\quad\quad\quad\quad\, {\color{blue}$\boldsymbol{\triangleright}$} Circulation
\State \hspace{0.3cm} \textbf{If} $k==n-1$, \textbf{then}
\State \hspace{0.6cm} Set ${\bs G}^{(k)} = {\bs I}$, and compute ${\bs g}^{(k)}$ by using model points $\{{\bs x}_i\}_{i=0}^{n-1}$;
\State \hspace{0.3cm} \textbf{Else}
\State \hspace{0.6cm} Go to the line \textbf{8};
\State \hspace{0.3cm} \textbf{Endif}
\State \hspace{0.3cm} Solve \eqref{eq2024122703} by using the GMRES;
\State \hspace{0.3cm} Compute $\triangle{\bs g}^{(k)}$ and $\triangle{\bs G}^{(k)}$ by \eqref{eq2024122503} and \eqref{eq2024122504};
\State \hspace{0.3cm} ${\bs g}^{(k)}$ and ${\bs G}^{(k)}$ is updated by \eqref{eq2024122303};   \quad\quad\;\; {\color{blue}$\boldsymbol{\triangleright}$} Update ${\bs g}$ and ${\bs G}$
\State \hspace{0.3cm} \textbf{If} $\|{\bs g}^{(k)}_{k}\|_2< \epsilon$, \textbf{then}
\State \hspace{0.6cm} ${\bs x}^{*} \leftarrow {\bs x}_{k}$, and break;
\State \hspace{0.3cm} \textbf{Else}
\State \hspace{0.3cm} Compute ${\bs x}_{k+1} = {\bs x}_{k}$ - ${\bs G}^{(k)}$$\backslash {\bs g}^{(k)}_{k}$;  \quad\quad\;\quad\hspace{-0.05cm} {\color{blue}$\boldsymbol{\triangleright}$} Newton iteration
\State \hspace{0.3cm} Update $\{{\bs x}_i\}_{i=k-n+2}^{k+1}$ by deleting ${\bs x}_{k-n+1}$, and go to the line \textbf{2}.
\State \hspace{0.3cm} \textbf{Endif}
\State \textbf{End}
\State Return ${\bs x}^{*}$
\end{algorithmic}
\end{breakablealgorithm}
\end{minipage}
\end{center}
To be specific,
the main purpose in the step 1 is to initialize model points $\{{\bs x}_i\}_{i = k-n+1}^{k}$. In the step 2, ${\bs g}^{(k)}$ and ${\bs G}^{(k)}$ shown in \eqref{eq2024122301} is updated by \eqref{eq2024122505} or \eqref{eq20241227001}. In the step 3, the classical Newton iteration is used to update model points, i.e., ${\bs x}_{k+1} = {\bs x}_{k} - {\bs G}^{(k)}\backslash{\bs g}^{(k)}_k$, and ${\bs x}_{k-n+1}$ is replaced by the point ${\bs x}_{k+1}$. Next,
${\bs g}^{(k+1)}$ and ${\bs G}^{(k+1)}$ will be computed based on model points $\{{\bs x}_i\}_{i = k-n+2}^{k+1}$. Finally, a solution of the unconstrained optimization problem \eqref{eq2024121301} is obtained by repeating the steps 2-3. For clarity, the computing process corresponding to Fig.\ref{figure2025011301} is summarized in the \textbf{Algorithm} \ref{algorithm2025011301}.

Since the simplified $\triangle {\bs g}_{*}^{(k)}$ and $\triangle {\bs G}_{*}^{(k)}$ shown in \eqref{eq2025010102} is first proposed and derived, we should state how they are used based on the \textbf{Algorithm} \ref{algorithm2025011301}. In addition, more details for the \textbf{Algorithm} \ref{algorithm2025011301} also need to be addressed. Therefore, some remarks are listed as follows:

{\em \textbf{In the line 8}:} When \eqref{eq20241227001} is used to update $\triangle{\bs g}^{(k)}_{*}$ and $\triangle{\bs G}^{(k)}_{*}$ (or $\triangle{\bs g}^{(k)}$ and $\triangle{\bs G}^{(k)}$), $\eta_{k}$ and $\theta_{k}$ are updated by using \eqref{eq2025012001}. Compared with \eqref{eq2024122505}, the simplified constrained model \eqref{eq20241227001} has a distinct advantage, i.e., it is not necessary to solve the linear system \eqref{eq2024122703}. This greatly improves the computational efficiency.

{\em \textbf{In the line 9}:} For the simplified constrained model \eqref{eq20241227001}, $\triangle{\bs g}^{(k)}$ and $\triangle{\bs G}^{(k)}$ are directly updated by using \eqref{eq2025010202}. Compared with \eqref{eq2024122505}, this not only simplifies the computational complexity, but also improves their accuracy. In addition, the parameter $\nu_{k+1}$ in \eqref{eq2025010202} is changed as follow:
\begin{equation*}
\nu_{k+1}=
\left\{
\begin{aligned}
&1.1\nu_{k},\quad \text{\ if\ }\;\|{\bs V}^{\top}_1\triangle {\bs g}_{*}^{(k)}\|_2^2 \ge 1.1 \|{\bs W}\triangle {\bs G}_{*}^{(k)}\|_{F}^2,\\
&\nu_{k},\quad \quad \,
 \text{\ if\ }\;0.9 \|{\bs W}\triangle {\bs G}_{*}^{(k)}\|_{F}^2\le \|{\bs V}^{\top}_1\triangle {\bs g}_{*}^{(k)}\|_2^2 \le 1.1 \|{\bs W}\triangle {\bs G}_{*}^{(k)}\|_{F}^2,\\
&0.9\nu_{k}, \quad\text{\ otherwise},
\end{aligned}
\right.
\end{equation*}
where the initial value of the parameter $\nu_{k}$ is equal to 1.

{\em \textbf{In the line 14}:} The classical Newton iteration is used. As we known, different quasi-Newton methods are often used to update the iteration point. Here they can also be extended in the \textbf{Algorithm} \ref{algorithm2025011301}. To be specific, when the rank-one quasi-Newton method is considered, the famous Sherman-Morrison formula should be used as follow:
\begin{equation}
({\bs G}+{\bs u}{\bs v}^{\top})^{-1} = {\bs G}^{-1} - \frac{{\bs G}^{-1}{\bs u}{\bs v}^{\top}{\bs G}^{-1}}{1+{\bs v}^{\top}{\bs G}^{-1}{\bs u}}.
\end{equation}
In addition, for a given initial guess ${\bs x}_{k-1} (k := n)$ (see the line 1), we first obtain ${\bs x}_{k}$ by using the steepest descent method. Then, we denote
\begin{equation}
{\bs p}^{(k-1)} := {\bs x}_{k} - {\bs x}_{k-1}  \quad\quad
{\bs q}^{(k-1)} : = {\bs g}^{(k)}_k - {\bs g}^{(k-1)}_{k-1},
\end{equation}
and the corresponding iteration becomes
\begin{equation}
{\bs B}^{(k)} = {\bs B}^{(k-1)} + \frac{({\bs p}^{(k-1)}-{\bs B}^{(k-1)}{\bs q}^{(k-1)})({\bs p}^{(k-1)})^{\top}{\bs B}^{(k-1)}}{({\bs p}^{(k-1)})^{\top}{\bs B}^{(k-1)}{\bs q}^{(k-1)}}
\end{equation}
\begin{equation}
\hspace{-1cm}{\bs x}_{k+1} = {\bs x}_{k} - {\bs B}^{(k)}{\bs g}^{(k)}_k
\end{equation}
where ${\bs B}^{(k)} = ({\bs G}^{(k)})^{-1}$. On the other hand, the rank-two quasi-Newton method (i.e., BFGS) can also be used. Here we don't present it, and the reader can referred to \cite{sun2006optimization} for more details.

\section{Numerical results}\label{sect5}
In this section, some numerical experiments involving smooth, derivative blasting and non-smooth problems are tested to show the efficiency of the \textbf{Algorithm} \ref{algorithm2025011301}, where we focus on updating ${\bs g}^{(k)}$ and ${\bs G}^{(k)}$ by using \eqref{eq20241227001}
, and all programs are carried out on a MacBook Pro with an Apple M3 Max chip and 64 GB of memory, running macOS 14.5\,(23F79).\vspace{0.2cm}

\hspace{-0.45cm}\textbf{Case 1. smooth problems}
\medskip

Five smooth problems are considered to illustrate the efficiency of the Algorithm \ref{algorithm2025011301}, and these problems can be found on the World-Wide Web at \href{https://www.sfu.ca/~ssurjano/}{\texttt{https://al-roomi.org/benchmarks/unconstrained}}

\begin{table}[h]\small
\centering
\setlength{\tabcolsep}{3.5pt}
\caption{Performance of \textbf{Algorithm} \ref{algorithm2025011301} and  other methods.}
\label{Table2025020701}
\begin{tabular}{llcccccccc}
\toprule
& {}           & \multicolumn{2}{c}{{Trust-Region}} & \multicolumn{2}{c}{{DFO-TR}}& \multicolumn{2}{c}{{Newton method}}  & \multicolumn{2}{c}{{\textbf{Algorithm} \ref{algorithm2025011301}}} \\
\cmidrule(lr){3-4} \cmidrule(lr){5-6} \cmidrule(lr){7-8} \cmidrule(lr){9-10}
             &  IG     & Error & CPU             & Error & CPU              & Error & CPU             & Error & CPU             \\
\midrule
{Woods}
&$\text{IG}_1$       & $6.34 e{-8}$ & 8.52             & $5.18 e{-8}$ & 9.05              & $4.92 e{-8}$ & 9.84             & $6.28 e{-9}$ & 6.13             \\
problem shown &$\text{IG}_2$       & $7.83 e{-8}$ & 8.76             & $6.25 e{-8}$ & 8.95              & $5.94 e{-8}$ & 9.21             & $6.09 e{-9}$ & 5.93             \\
in  \cite{conn1988testing} & $\text{IG}_3$       & $8.07 e{-8}$ & 9.03             & $7.52 e{-8}$ & 9.11              & $6.63 e{-8}$ & 9.54             & $5.92 e{-9}$ & 5.87             \\
\midrule
{Rosenbrock}  &$\text{IG}_1$       & $6.12 e{-8}$ & 8.64             & $5.26 e{-8}$ & 8.89              & $4.95 e{-8}$ & 9.23             & $6.19 e{-9}$ & 6.01             \\
problem shown &$\text{IG}_2$       & $7.58 e{-8}$ & 8.91             & $6.83 e{-8}$ & 9.03              & $6.04 e{-8}$ & 9.51             & $5.97 e{-9}$ & 5.73             \\
in \cite{Luksan2010} & $\text{IG}_3$       & $8.11 e{-8}$ & 9.24             & $7.63 e{-8}$ & 9.34              & $6.87 e{-8}$ & 9.62             & $5.89 e{-9}$ & 5.62             \\
\midrule
Sum of \(\lfloor n/4\rfloor\) & $\text{IG}_1$       & $8.12 e{-8}$ & 10.24            & $7.15 e{-8}$ & 10.52             & $6.98 e{-8}$ & 11.01            & $7.02 e{-9}$ & 6.41             \\
shown   & $\text{IG}_2$       & $9.34 e{-8}$ & 10.61            & $8.04 e{-8}$ & 10.78             & $7.92 e{-8}$ & 11.43            & $8.05 e{-9}$ & 6.12             \\
in \cite{J-more-test}  & $\text{IG}_3$       & $9.88 e{-8}$ & 11.05            & $8.97 e{-8}$ & 11.13             & $8.53 e{-8}$ & 11.74            & $8.19 e{-9}$ & 6.24             \\
\midrule
Sparse & $\text{IG}_1$       & $7.84 e{-8}$ & 9.84             & $6.72 e{-8}$ & 10.12             & $6.34 e{-8}$ & 10.78            & $6.53 e{-9}$ & 5.97             \\
 problem  & $\text{IG}_2$       & $8.96 e{-8}$ & 10.01            & $7.64 e{-8}$ & 10.25             & $7.21 e{-8}$ & 10.93            & $7.18 e{-9}$ & 5.79             \\
shown in  \cite{CUTEr}
 & $\text{IG}_3$       & $9.45 e{-8}$ & 10.54            & $8.34 e{-8}$ & 10.61             & $8.01 e{-8}$ & 11.31            & $7.29 e{-9}$ & 5.84             \\
\midrule
Dixon-Maany & $\text{IG}_1$       & $8.54 e{-8}$ & 10.41            & $7.32 e{-8}$ & 10.67             & $6.94 e{-8}$ & 11.24            & $6.91 e{-9}$ & 6.02             \\
problem & $\text{IG}_2$       & $9.32 e{-8}$ & 10.62            & $8.15 e{-8}$ & 10.93             & $7.74 e{-8}$ & 11.63            & $7.25 e{-9}$ & 5.85             \\
shown in \cite{CUTEr}  & $\text{IG}_3$       & $9.91 e{-8}$ & 11.03            & $8.83 e{-8}$ & 11.24             & $8.31 e{-8}$ & 12.02            & $7.37 e{-9}$ & 5.76             \\
\bottomrule
\end{tabular}
\end{table}

\begin{table}[!h]\small
\centering
\setlength{\tabcolsep}{3.5pt}
\caption{The verification of \eqref{eq2025012205} for smooth problems shown in table \ref{Table2025020701}.}
\label{Table2025020702}
\begin{tabular}{lcccccccc}
\toprule
 {}           & \multicolumn{2}{c}{{200-th step}} & \multicolumn{2}{c}{{300-th step}}& \multicolumn{2}{c}{{400-th step}}  & \multicolumn{2}{c}{{500-th step}} \\
\cmidrule(lr){2-3} \cmidrule(lr){4-5} \cmidrule(lr){6-7} \cmidrule(lr){8-9}
IG     & $\mathcal{E}^{(1)}$ &  $\mathcal{E}^{(2)}$            &  $\mathcal{E}^{(1)}$&  $\mathcal{E}^{(2)}$         &  $\mathcal{E}^{(1)}$ &  $\mathcal{E}^{(2)}$            &  $\mathcal{E}^{(1)}$ &  $\mathcal{E}^{(2)}$         \\
\midrule
\multicolumn{9}{c}{{Woods} problem shown in  \cite{conn1988testing}}\\
\midrule
$\text{IG}_1$       & $6.34 e{4}$ & $5.92 e{-6}$   & $5.18 e{-7}$ & $4.97 e{-5}$   & $4.92 e{-8}$ & $4.64 e{-8}$   & $6.28 e{-9}$ & $5.91 e{-9}$  \\
$\text{IG}_2$  & $7.83 e{-6}$ & $7.42 e{-3}$   & $6.25 e{-7}$ & $5.92 e{-7}$   & $5.94 e{-8}$ & $5.51 e{-8}$   & $6.09 e{-9}$ & $5.82 e{-9}$  \\
 $\text{IG}_3$       & $8.07 e{-6}$ & $7.85 e{-6}$   & $7.52 e{-7}$ & $7.19 e{-10}$   & $6.63 e{-8}$ & $6.29 e{-8}$   & $5.92 e{-9}$ & $5.67 e{-9}$  \\
\midrule
\multicolumn{9}{c}{{Rosenbrock} problem shown in \cite{Luksan2010}}\\
\midrule
  $\text{IG}_1$       & $6.12 e{-6}$ & $5.83 e{-6}$   & $5.26 e{-7}$ & $4.97 e{-7}$   & $4.95 e{-8}$ & $4.64 e{-8}$   & $6.19 e{-9}$ & $5.91 e{-9}$  \\
 $\text{IG}_2$       & $7.58 e{-6}$ & $7.32 e{-6}$   & $6.83 e{-12}$ & $6.52 e{-7}$   & $6.04 e{-8}$ & $5.73 e{-8}$   & $5.97 e{-9}$ & $5.73 e{-10}$  \\
  $\text{IG}_3$       & $8.11 e{-6}$ & $7.93 e{-6}$   & $7.63 e{-7}$ & $7.39 e{-7}$   & $6.87 e{-8}$ & $6.61 e{-8}$   & $5.89 e{-9}$ & $5.62 e{-11}$  \\
\midrule
\multicolumn{9}{c}{Sum of \(\lfloor n/4\rfloor\) shown in \cite{J-more-test} }\\
\midrule
 $\text{IG}_1$       & $8.12 e{-6}$ & $7.93 e{-6}$   & $7.15 e{-7}$ & $6.98 e{-7}$   & $6.98 e{-8}$ & $6.75 e{-8}$   & $7.02 e{-9}$ & $6.81 e{-9}$  \\
$\text{IG}_2$       & $9.34 e{-6}$ & $9.12 e{-6}$   & $8.04 e{-7}$ & $7.82 e{-7}$   & $7.92 e{-8}$ & $7.64 e{-6}$   & $8.05 e{-9}$ & $7.81 e{-9}$  \\
$\text{IG}_3$       & $9.88 e{-6}$ & $9.65 e{-6}$   & $8.97 e{-7}$ & $8.72 e{-7}$   & $8.53 e{-8}$ & $8.21 e{-8}$   & $8.19 e{-9}$ & $7.94 e{-5}$  \\
\midrule
\multicolumn{9}{c}{Sparse  problem shown in  \cite{CUTEr} }\\
\midrule
$\text{IG}_1$       & $7.84 e{-5}$ & $7.61 e{-10}$   & $6.72 e{-7}$ & $6.51 e{-7}$   & $6.34 e{-8}$ & $6.13 e{-8}$   & $6.53 e{-9}$ & $6.32 e{-9}$  \\
 $\text{IG}_2$       & $8.96 e{-6}$ & $8.75 e{-6}$   & $7.64 e{-7}$ & $7.42 e{-7}$   & $7.21 e{-8}$ & $6.98 e{-8}$   & $7.18 e{-9}$ & $6.97 e{-10}$  \\
 $\text{IG}_3$       & $9.45 e{-6}$ & $9.21 e{-6}$   & $8.34 e{-8}$ & $8.11 e{-7}$   & $8.01 e{-8}$ & $7.75 e{-8}$   & $7.29 e{-9}$ & $7.07 e{-12}$  \\
\midrule
\multicolumn{9}{c}{Dixon-Maany problem shown in \cite{CUTEr}}\\
\midrule
$\text{IG}_1$       & $8.54 e{-4}$ & $8.31 e{-6}$   & $7.32 e{-7}$ & $7.11 e{-7}$   & $6.94 e{-8}$ & $6.73 e{-8}$   & $6.91 e{-9}$ & $6.72 e{-9}$  \\
$\text{IG}_2$       & $9.32 e{-6}$ & $9.11 e{-6}$   & $8.15 e{-7}$ & $7.92 e{-10}$   & $7.74 e{-8}$ & $7.51 e{-8}$   & $7.25 e{-9}$ & $7.03 e{-10}$  \\
 $\text{IG}_3$       & $9.91 e{-6}$ & $9.68 e{-6}$   & $8.83 e{-7}$ & $8.57 e{-8}$   & $8.31 e{-8}$ & $8.03 e{-8}$   & $7.37 e{-9}$ & $7.14 e{-11}$  \\
\bottomrule
\end{tabular}
\end{table}
\hspace{-0.4cm}or \href{https://www.sfu.ca/~ssurjano/}{\texttt{https://www.sfu.ca/\raisebox{-0.9ex}{\texttt{\char`~}}ssurjano/}}. When the Algorithm \ref{algorithm2025011301} is updated to the 500-th step, corresponding numerical results are presented in Tables \ref{Table2025020701}-\ref{Table2025020702}, where CPU represents the computation time,
$\mathcal{E}^{(r)} := \|(\mathcal{E}^{(r)}_{k-n+1}, \cdots, \mathcal{E}^{(r)}_{k-1})\|_{\infty}\, (r = 1, 2)$ and IG denotes the initial guess. Moreover, different initial guesses are also considered, i.e.,
\begin{equation}
\text{IG}_1=(1,\cdots,1)^{\top},\quad\quad \text{IG}_2=\sin(\text{IG}_1),  \quad\quad \text{IG}_3=\exp(\text{IG}_1).
\end{equation}
From Table \ref{Table2025020701}, it can be concluded that
the Algorithm \ref{algorithm2025011301} has distinct advantages over other methods (i.e., the classical Trust-Region method, derivative-free Trust-Region method (DFO-TR) shown in \cite{conn1997algorithm}, the classical Newton method), where the computation time for the Algorithm \ref{algorithm2025011301} is less than one for other methods, while the Algorithm \ref{algorithm2025011301} has higher accuracy degree. In Table \ref{Table2025020702}, the verification of \eqref{eq2025012205} for these smooth problems
is shown, where $\mathcal{E}^{(r)} (r = 1, 2)$ tend to zero with the iteration. This agrees well with \eqref{eq2025012205}.

\vspace{0.2cm}
\hspace{-0.45cm}\textbf{Case 2. problems with derivative blasting}
\medskip

To the best of our knowledge, there exist some problems with large derivative (i.e., derivative blasting). Since the large derivative can easily cause numerical instability, computing them will encounter inherent difficulties. Here the Algorithm \ref{algorithm2025011301} is also used to compute some problems with derivative blasting as follows:

\medskip

\textbf{Problem 1:}
\begin{equation}\label{eq2025021002}
\begin{aligned}
\hspace{-0.7cm}f({\bs x}) =
\left[ 1000 \left( x_2 - x_1^2 \right) \right]^2 + 1000(1 - x_1)^2 +
90000 \left( x_4 - x_3^2 \right)^2 + 1000(1 - x_3)^2 + \\
10100 \left[ (x_2 - 1)^2 + 1000(x_4 - 1)^2 \right] +
19800 (x_2 - 1)(x_4 - 1)+\sum_{i=1}^{5}x_i^2,
\end{aligned}
\end{equation}

\textbf{Problem 2:}
\begin{equation}\label{eq2025021003}
f({\bs x}) = 1000\sum_{i=1}^{100} \left[ 100 (x_{i+1} - x_i^2)^2 + (1 - x_i)^2 \right],
\end{equation}

\textbf{Problem 3:}
\begin{equation}\label{eq2025021004}
f({\bs x}) = 10^{5} \sum_{i=1}^{100} \cos(5\pi x_i) - 10^{3}\sum_{i=1}^{100} x_i^2.
\end{equation}
It is worth pointing out that the exact solution for \eqref{eq2025021002}-\eqref{eq2025021004} is all $\hat{{\bs x}} = (1, \cdots, 1)^{\top}$, and $f(\hat{{\bs x}}) = 0$. In addition, as seen in Fig.\ref{NewExample61}, $\|\nabla f\|_{\infty}$ is very large. Numerical results obtained by the Algorithm \ref{algorithm2025011301} are presented in Tables \ref{Table2025021201}-\ref{Table2025021202}. In Table \ref{Table2025021201}, our proposed method has a good performance. Numerical results shown in Table \ref{Table2025021202} verify the correctness of \eqref{eq2025012205} again.

\begin{figure}[!h]
\begin{center}
\hspace{-0.4cm}\subfigure[$\|\nabla f\|_{\infty}$ vs $x_{1}$ for the \textbf{problem 1}]{ \includegraphics[width=10.5cm,height=6.3cm]{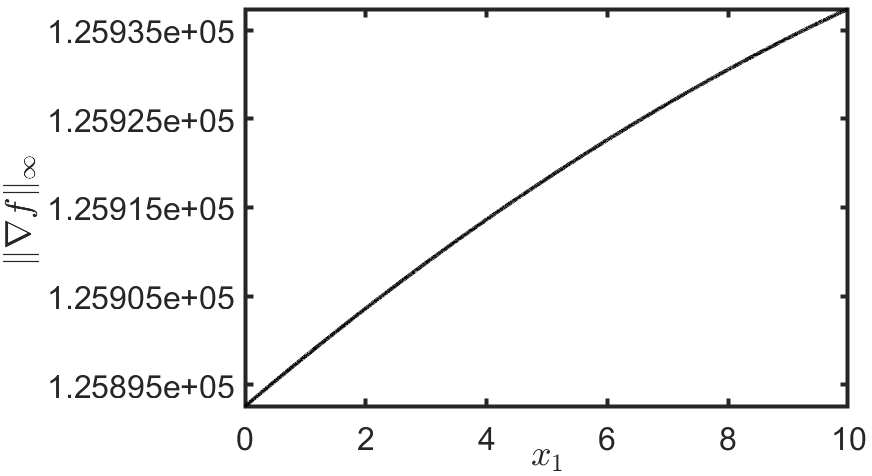}}\\\vspace{0.2cm}
\hspace{-0.05cm}\subfigure[$\|\nabla f\|_{\infty}$ vs $x_{1}$ for the \textbf{problem 2}]{ \includegraphics[width=5.5cm,height=4cm]{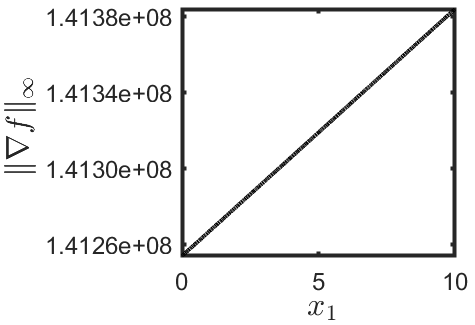}}\;\quad\quad
\subfigure[$\|\nabla f\|_{\infty}$ vs $x_{1}$ for the \textbf{problem 3}]{ \includegraphics[width=5.5cm,height=4cm]{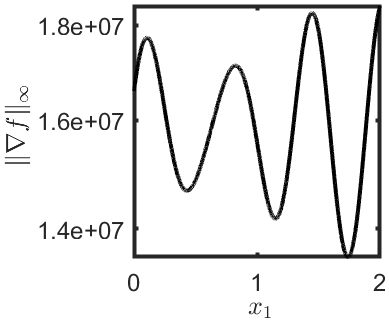}}\;
\caption{Some problems with derivative blasting.}\label{NewExample61}
\end{center}
\end{figure}

\begin{table}[h!]\small
\centering
\setlength{\tabcolsep}{5.5pt}
\caption{\hspace{-0.15cm}Performance of \textbf{Algorithm} \ref{algorithm2025011301} and other methods for \eqref{eq2025021002}-\eqref{eq2025021004}.}
\label{Table2025021201}
\hspace{-0.3cm}\begin{tabular}{clcccccccc}
\toprule[0.03cm]
& {}           & \multicolumn{2}{c}{{Trust-Region}} & \multicolumn{2}{c}{{DFO-TR}}& \multicolumn{2}{c}{{Newton method}}  & \multicolumn{2}{c}{\textbf{Algorithm} \ref{algorithm2025011301}} \\
\cmidrule(lr){3-4} \cmidrule(lr){5-6} \cmidrule(lr){7-8} \cmidrule(lr){9-10}
             &  IG     & Error & CPU             & Error & CPU              & Error & CPU             & Error & CPU             \\
\midrule[0.02cm]
&$\text{IG}_1$       & $6.81 e{-8}$ & 12.45             & $5.24 e{-8}$ & 13.12              & $4.87 e{-8}$ & 13.91             & $6.32 e{-9}$ & 10.17             \\
 {\bf Problem 1}     &$\text{IG}_2$      & $7.77 e{-8}$ & 12.82             & $6.31 e{-8}$ & 13.02              & $6.02 e{-8}$ & 13.29             & $6.02 e{-9}$ & 9.98             \\
 & $\text{IG}_3$       & $8.12 e{-8}$ & 13.09             & $7.48 e{-8}$ & 13.20              & $6.69 e{-8}$ & 13.62             & $5.89 e{-9}$ & 9.81             \\
\midrule[0.02cm]
  &$\text{IG}_1$       & $6.15 e{-8}$ & 12.69             & $5.29 e{-8}$ & 12.91              & $4.98 e{-8}$ & 13.29             & $6.21 e{-9}$ & 10.04             \\
 {\bf Problem 2}  &$\text{IG}_2$     & $7.62 e{-8}$ & 12.94             & $6.79 e{-8}$ & 13.10              & $6.07 e{-8}$ & 13.56             & $5.94 e{-9}$ & 9.76             \\
 & $\text{IG}_3$       & $8.09 e{-8}$ & 13.21             & $7.67 e{-8}$ & 13.41              & $6.91 e{-8}$ & 13.68             & $5.86 e{-9}$ & 9.64             \\
\midrule[0.02cm]
 & $\text{IG}_1$       & $8.09 e{-8}$ & 14.31            & $7.12 e{-8}$ & 14.45             & $6.94 e{-8}$ & 15.07            & $7.05 e{-9}$ & 10.45             \\
 {\bf Problem 3}  & $\text{IG}_2$       & $9.29 e{-8}$ & 14.55            & $8.09 e{-8}$ & 14.81             & $7.87 e{-8}$ & 15.49            & $8.01 e{-9}$ & 10.18             \\
 & $\text{IG}_3$       & $9.91 e{-8}$ & 15.12            & $8.94 e{-8}$ & 15.21             & $8.47 e{-8}$ & 15.69            & $8.23 e{-9}$ & 10.28             \\
\bottomrule[0.03cm]
\end{tabular}
\end{table}

\begin{table}[h!]\small
\centering
\setlength{\tabcolsep}{5.5pt}
\caption{The verification of \eqref{eq2025012205} for \eqref{eq2025021002}-\eqref{eq2025021004}.}
\label{Table2025021202}
\begin{tabular}{lcccccccc}
\toprule
 {}           & \multicolumn{2}{c}{{700-th step}} & \multicolumn{2}{c}{{800-th step}}& \multicolumn{2}{c}{{900-th step}}  & \multicolumn{2}{c}{{1000-th step}} \\
\cmidrule(lr){2-3} \cmidrule(lr){4-5} \cmidrule(lr){6-7} \cmidrule(lr){8-9}
             IG     & $\mathcal{E}^{(1)}$ &  $\mathcal{E}^{(2)}$            &  $\mathcal{E}^{(1)}$&  $\mathcal{E}^{(2)}$         &$\mathcal{E}^{(1)}$ &  $\mathcal{E}^{(2)}$            &$\mathcal{E}^{(1)}$ &  $\mathcal{E}^{(2)}$          \\
\midrule
\multicolumn{9}{c}{\bf Problem 1}\\
\midrule
$\text{IG}_1$       & $3.78 e{-4}$ & $5.95 e{-6}$   & $5.12 e{-7}$ & $4.91 e{-5}$   & $4.86 e{-8}$ & $4.72 e{-8}$   & $6.31 e{-9}$ & $5.97 e{-9}$  \\
  $\text{IG}_2$  & $7.91 e{-6}$ & $7.35 e{-3}$   & $6.19 e{-7}$ & $5.87 e{-7}$   & $5.87 e{-8}$ & $5.47 e{-8}$   & $6.04 e{-9}$ & $5.79 e{-9}$  \\
 $\text{IG}_3$       & $8.12 e{-6}$ & $7.92 e{-6}$   & $7.46 e{-7}$ & $7.10 e{-10}$   & $6.69 e{-8}$ & $6.31 e{-8}$   & $5.88 e{-9}$ & $5.70 e{-9}$  \\
\midrule
\multicolumn{9}{c}{\bf Problem 2}\\
\midrule
  $\text{IG}_1$       & $6.09 e{-6}$ & $5.79 e{-6}$   & $5.21 e{-7}$ & $4.92 e{-7}$   & $4.91 e{-8}$ & $4.71 e{-8}$   & $6.14 e{-9}$ & $5.88 e{-9}$  \\
 $\text{IG}_2$       & $7.65 e{-6}$ & $7.29 e{-6}$   & $6.78 e{-12}$ & $6.49 e{-7}$   & $6.02 e{-8}$ & $5.76 e{-8}$   & $5.91 e{-9}$ & $5.71 e{-10}$  \\
  $\text{IG}_3$       & $8.19 e{-6}$ & $7.87 e{-6}$   & $7.69 e{-7}$ & $7.34 e{-7}$   & $6.91 e{-8}$ & $6.65 e{-8}$   & $5.92 e{-9}$ & $5.65 e{-11}$  \\
\midrule
\multicolumn{9}{c}{\bf Problem 3}\\
\midrule
 $\text{IG}_1$       & $8.14 e{-6}$ & $7.95 e{-6}$   & $7.19 e{-7}$ & $7.01 e{-7}$   & $7.01 e{-8}$ & $6.79 e{-8}$   & $7.08 e{-9}$ & $6.85 e{-9}$  \\
$\text{IG}_2$       & $9.38 e{-6}$ & $9.14 e{-6}$   & $8.07 e{-7}$ & $7.85 e{-7}$   & $7.88 e{-8}$ & $7.59 e{-6}$   & $8.09 e{-9}$ & $7.79 e{-9}$  \\
$\text{IG}_3$       & $9.92 e{-6}$ & $9.67 e{-6}$   & $9.01 e{-7}$ & $8.78 e{-7}$   & $8.57 e{-8}$ & $8.18 e{-8}$   & $8.22 e{-9}$ & $7.97 e{-11}$  \\
\bottomrule
\end{tabular}
\end{table}

\clearpage
\hspace{-0.35cm}\textbf{Case 3. nonsmooth problems}
\medskip

In many scientific and engineering problems, nonsmooth problems
are often seen, such as regression \cite{1996Tibshirani}, compressed sensing \cite{2006Donoho}, visual coding \cite{1996Olshausen}, imaging decomposition \cite{2015Soubies}, etc. Here the Algorithm \ref{algorithm2025011301} is used to test some nonsmooth problems, and their corresponding objective functions shown in \cite{1989Krzysztof, 1992Makela} are as follows:
\begin{equation}\label{eq2025022104}
\vspace{0.15cm}
\hspace{-3.9cm}  (\textrm{P1}):  \hspace{2cm}  f({\bs x}) = \max_{1 \leq i \leq 50} | \sum_{j=1}^{50} \frac{x_j}{i + j - 1}|,
\end{equation}
\begin{equation}\label{eq2025022105}
\vspace{0.15cm}
\hspace{-3.9cm} (\textrm{P2}):  \hspace{1.4cm}\quad\quad\quad  f({\bs x}) = \sum_{i=1}^{50}|\sum_{j=1}^{50} \frac{x_j}{i + j - 1}|,
\end{equation}
\begin{equation}\label{eq2025022106}
\vspace{0.1cm}
(\textrm{P3}):  \quad  f({\bs x}) = \max \left\{ x_1^2 + (x_2 - 1)^2 + x_2 - 1, -x_1^2 - (x_2 - 1)^2 + x_2 + 1 \right\},
\end{equation}
where the exact solution of \eqref{eq2025022104}-\eqref{eq2025022106} is all $(0, \cdots, 0)^{\top}$. In addition, the nonsmooth sparse regression problem with cardinality penalty shown in \cite{2020BianChen} is also considered, i.e.,\vspace{0.13cm}
\begin{equation}\label{eq2025022101}
\hspace{-3.2cm}(\textrm{P4}): \quad\quad\quad\quad\quad\quad \min_{x \in \mathbb{R}^n} f({\bs x}) := \|A{\bs x} - b\|_1 + \lambda \|{\bs x}\|_0,\vspace{0.1cm}
\end{equation}
where $A\in \mathbb{R}^{m\times n}, b \in \mathbb{R}^n$, $0 \leq x_i \leq 1$ and
\begin{equation*}
\|{\bs x}\|_0 = \sum\limits_{i=1,\, x_i\neq 0}^n|x_i|^0.
\end{equation*}
Moreover, the corresponding continuous relaxation problem is also considered, i.e.,
\begin{equation}\label{eq2025022103}
\min_{{\bs x} \in \mathbb{R}^n} \|A{\bs x} - b\|_1 + \lambda \Phi({\bs x}),
\end{equation}
where
\begin{equation}\label{eq2025022102}
\phi(t) = \min \{ 1, |t| / \mu \}   \quad\, \textrm{and} \quad\,
\Phi({\bs x}) = \sum_{i=1}^{n} \phi(x_i).
\end{equation}
As seen in Fig.\ref{NewExample63}, $\phi$ in \eqref{eq2025022102} is plotted with varying $t$, indicating that $\phi$ is a nonsmooth function. This means that $\Phi({\bs x})$ in \eqref{eq2025022103} is also nonsmooth. Moreover, as mentioned in \cite{2020BianChen}, the solution of \eqref{eq2025022103} will approximate the solution of \eqref{eq2025022102} with $\mu \rightarrow 0$, which will be verified by the Algorithm \ref{algorithm2025011301}.\vspace{-0.3cm}

\begin{figure}[!h]
\begin{center}
\includegraphics[width=5.5cm]{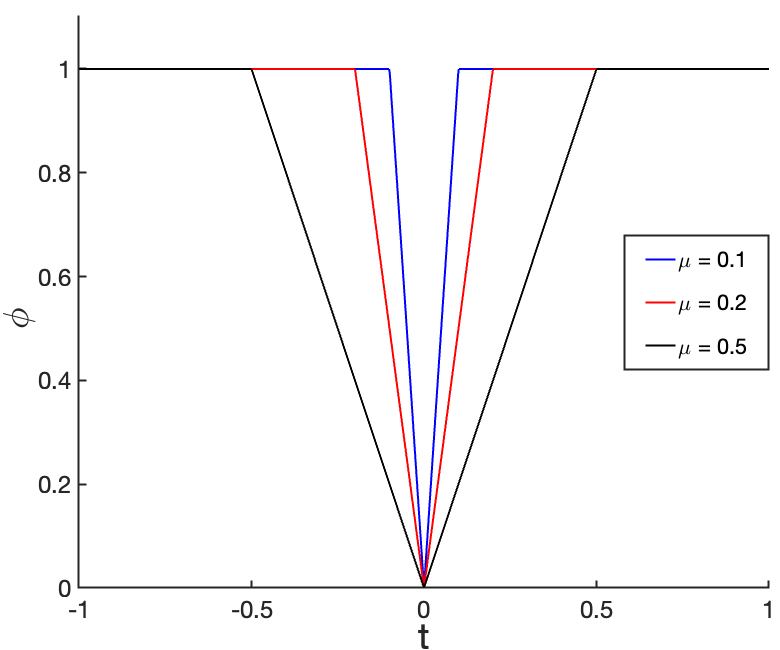}\vspace{-0.2cm}
\caption{$\phi$ vs $t$ in \eqref{eq2025022102}.}\label{NewExample63}
\end{center}
\end{figure}
In \eqref{eq2025022101} or \eqref{eq2025022103}, we choose $\lambda = 1$, $b = 1$ and $A = (1, 1)$. In Table.\ref{Table2025021201}, the performances of Algorithm \ref{algorithm2025011301} and other methods are presented, where some methods don't converge (denoted by $\times$) with different initial guesses, while the computation time (CPU) for the Algorithm \ref{algorithm2025011301} is less than one for the DFO-TR, and it has higher accuracy. In Table \ref{Table2025022108}, the validity of \eqref{eq2025012205} for \eqref{eq2025022104}-\eqref{eq2025022101} is verified based on the Algorithm \ref{algorithm2025011301}. In Table \ref{table2025022101}, $\|{\bs x}^*\|_{\infty}$ is reduced with $\mu\rightarrow 0$, where ${\bs x}^*$ represents the solution of \eqref{eq2025022103}. This agrees well with the result mentioned in \cite{2020BianChen}, indicating that our proposed method is very feasible and efficient.

\begin{table}[h!]\small
\centering
\setlength{\tabcolsep}{4pt}
\caption{Performance of \textbf{Algorithm} \ref{algorithm2025011301} and other methods for \eqref{eq2025022104}-\eqref{eq2025022101}.}
\label{Table2025021201}
\begin{tabular}{clcccccccc}
\toprule[0.03cm]
& {}           & \multicolumn{2}{c}{{Trust-Region}} & \multicolumn{2}{c}{{DFO-TR}}& \multicolumn{2}{c}{{Newton method}}  & \multicolumn{2}{c}{\textbf{Algorithm} \ref{algorithm2025011301}} \\
\cmidrule(lr){3-4} \cmidrule(lr){5-6} \cmidrule(lr){7-8} \cmidrule(lr){9-10}
             &\,  IG     & Error & CPU             & Error & CPU              & Error & CPU             & Error & CPU             \\
\midrule[0.02cm]
&\, $\text{IG}_1$       & $\times$ & $\times$             & $5.47 e{-8}$ & 13.22              & $\times$ & $\times$             & $6.50 e{-9}$ & 10.12             \\
(P1)\quad   &\, $\text{IG}_2$      & $\times$ & $\times$             & $6.59 e{-8}$ & 13.14              & $\times$ & $\times$             & $6.10 e{-9}$ & 9.85             \\
 &\, $\text{IG}_3$       & $\times$ & $\times$             & $\times$ & $\times$             & $\times$ & $\times$             & $5.97 e{-9}$ & 9.71             \\
\midrule[0.02cm]
  &\, $\text{IG}_1$       & $\times$ & $\times$             & $5.42 e{-8}$ & 13.05              & $\times$ & $\times$             & $6.38 e{-9}$ & 9.97             \\
 (P2)  &\, $\text{IG}_2$     & $\times$ & $\times$             & $6.98 e{-8}$ & 13.24              & $\times$ & $\times$             & $6.05 e{-9}$ & 9.62             \\
 &\, $\text{IG}_3$       & $\times$ & $\times$             &$\times$ & $\times$             & $\times$ & $\times$             & $5.89 e{-9}$ & 9.55             \\
\midrule[0.02cm]
 &\, $\text{IG}_1$       & $\times$ & $\times$            & $7.35 e{-8}$ & 4.58             & $\times$ & $\times$            & $7.20 e{-9}$ & 1.38             \\
 (P3) &\, $\text{IG}_2$       & $\times$ & $\times$            & $8.25 e{-8}$ & 5.01             & $\times$ & $\times$            & $8.15 e{-9}$ & 1.10             \\
 &\, $\text{IG}_3$       & $\times$ & $\times$            & $9.10 e{-8}$ & 5.34             & $\times$ & $\times$            & $8.32 e{-9}$ & 1.22             \\
 \midrule[0.02cm]
 &\, $\text{IG}_1$       & $\times$ & $\times$            & $7.35 e{-8}$ & 14.58             & $\times$ & $\times$            & $7.20 e{-9}$ & 10.38             \\
  (P4)\quad  &\, $\text{IG}_2$       & $\times$ & $\times$            & $8.25 e{-8}$ & 15.01             & $\times$ & $\times$            & $8.15 e{-9}$ & 10.10             \\
 &\, $\text{IG}_3$       & $\times$ & $\times$            & $9.10 e{-8}$ & 15.34             & $\times$ & $\times$            & $8.32 e{-9}$ & 10.22             \\
\bottomrule[0.03cm]
\end{tabular}
\end{table}
\vspace{-0.2cm}
\begin{table}[htbp]\small
\centering
\setlength{\tabcolsep}{3pt}
\caption{The verification of \eqref{eq2025012205} for \eqref{eq2025022104}-\eqref{eq2025022101}.}
\label{Table2025022108}
\begin{tabular}{lcccccccc}
\toprule
 {}           & \multicolumn{2}{c}{{700-th step}} & \multicolumn{2}{c}{{800-th step}}& \multicolumn{2}{c}{{900-th step}}  & \multicolumn{2}{c}{{1000-th step}} \\
\cmidrule(lr){2-3} \cmidrule(lr){4-5} \cmidrule(lr){6-7} \cmidrule(lr){8-9}
             IG     & $\mathcal{E}^{(1)}$ &  $\mathcal{E}^{(2)}$            &  $\mathcal{E}^{(1)}$&  $\mathcal{E}^{(2)}$         &$\mathcal{E}^{(1)}$ &  $\mathcal{E}^{(2)}$            &$\mathcal{E}^{(1)}$ &  $\mathcal{E}^{(2)}$          \\
\midrule
\multicolumn{9}{c}{(P1)}\\
\midrule
$\text{IG}_1$       & $3.92 e{-4}$ & $6.10 e{-6}$   & $5.21 e{-7}$ & $5.02 e{-5}$   & $4.95 e{-8}$ & $4.78 e{-8}$   & $6.45 e{-9}$ & $6.02 e{-9}$  \\
  $\text{IG}_2$  & $8.12 e{-6}$ & $7.50 e{-3}$   & $6.35 e{-7}$ & $6.01 e{-7}$   & $6.01 e{-8}$ & $5.60 e{-8}$   & $6.18 e{-9}$ & $5.90 e{-9}$  \\
 $\text{IG}_3$       & $8.30 e{-6}$ & $8.10 e{-6}$   & $7.60 e{-7}$ & $7.25 e{-10}$   & $6.85 e{-8}$ & $6.45 e{-8}$   & $6.00 e{-9}$ & $5.85 e{-9}$  \\
\midrule[0.03cm]
\multicolumn{9}{c}{(P2)}\\
\midrule
  $\text{IG}_1$       & $6.25 e{-6}$ & $5.95 e{-6}$   & $5.35 e{-7}$ & $5.02 e{-7}$   & $5.05 e{-8}$ & $4.80 e{-8}$   & $6.30 e{-9}$ & $6.02 e{-9}$  \\
 $\text{IG}_2$       & $7.88 e{-6}$ & $7.40 e{-6}$   & $6.90 e{-12}$ & $6.60 e{-7}$   & $6.18 e{-8}$ & $5.92 e{-8}$   & $6.05 e{-9}$ & $5.78 e{-10}$  \\
  $\text{IG}_3$       & $8.35 e{-6}$ & $8.05 e{-6}$   & $7.85 e{-7}$ & $7.49 e{-7}$   & $7.15 e{-8}$ & $6.85 e{-8}$   & $6.10 e{-9}$ & $5.90 e{-11}$  \\
\midrule[0.03cm]
\multicolumn{9}{c}{(P3)}\\
\midrule
 $\text{IG}_1$       & $8.30 e{-6}$ & $8.10 e{-6}$   & $7.35 e{-7}$ & $7.12 e{-7}$   & $7.18 e{-8}$ & $6.90 e{-8}$   & $7.20 e{-9}$ & $7.00 e{-9}$  \\
$\text{IG}_2$       & $9.50 e{-6}$ & $9.20 e{-6}$   & $8.22 e{-7}$ & $8.05 e{-7}$   & $8.10 e{-8}$ & $7.80 e{-6}$   & $8.25 e{-9}$ & $7.90 e{-9}$  \\
$\text{IG}_3$       & $10.00 e{-6}$ & $9.75 e{-6}$   & $9.15 e{-7}$ & $8.92 e{-7}$   & $8.75 e{-8}$ & $8.30 e{-8}$   & $8.35 e{-9}$ & $8.05 e{-11}$  \\
\midrule[0.03cm]
\multicolumn{9}{c}{(P4)}\\
\midrule
 $\text{IG}_1$       & $8.30 e{-6}$ & $8.10 e{-6}$   & $7.35 e{-7}$ & $7.12 e{-7}$   & $7.18 e{-8}$ & $6.90 e{-8}$   & $7.20 e{-9}$ & $7.00 e{-9}$  \\
$\text{IG}_2$       & $9.50 e{-6}$ & $9.20 e{-6}$   & $8.22 e{-7}$ & $8.05 e{-7}$   & $8.10 e{-8}$ & $7.80 e{-6}$   & $8.25 e{-9}$ & $7.90 e{-9}$  \\
$\text{IG}_3$       & $10.00 e{-6}$ & $9.75 e{-6}$   & $9.15 e{-7}$ & $8.92 e{-7}$   & $8.75 e{-8}$ & $8.30 e{-8}$   & $8.35 e{-9}$ & $8.05 e{-11}$  \\
\bottomrule
\end{tabular}
\end{table}
\vspace{-0.2cm}
\begin{table}[h!]
\centering
\caption{$\|{\bs x}^*\|_{\infty}$ vs $\mu$ for \eqref{eq2025022103}}
\label{table2025022101}
\renewcommand{\arraystretch}{1.2} 
\begin{tabular}{c|ccc}
\hline
$\mu$  & $10^{-1}$ & $10^{-2}$ & $10^{-3}$ \\
\hline
$\|{\bs x}^*\|_{\infty}$  & $1.87\times 10^{-3}$ & $2.73\times 10^{-6}$ & $7.36 \times 10^{-10}$ \\
\hline
\end{tabular}
\end{table}

Finally, a nonsmooth problem arising from the following logarithmic Schr\"{o}dinger equation in quantum mechanics (see \cite{2016Troy}) is first considered, i.e.,
\begin{equation}\label{2024072501}
\begin{cases}\vspace{0.15cm}
u'' + \frac{s-1}{r}u' + u\ln|u| = 0,    \quad\quad\quad  r \in (0, \infty),  \\
u'(0) = 0,   \quad\quad   (u(r), u'(r)) \to (0, 0) \; \textrm{as}\; r \to \infty,
\end{cases}
\end{equation}
where $s>1$. Obviously, the nonlinear term $u\ln|u|$ in \eqref{2024072501} exists a singularity at the origin. In
addition, it is worth pointing out that if $u$ is a solution of \eqref{2024072501}, so is $-u$. For simplicity, we truncate
the unbounded domain [0, $+\infty$) to the bounded domain $[0, b]$, where $b$ is a constant to be selected. Based on the Legendre-Galerkin method shown in \cite{shen2011spectral, 2022Two} and the least-square method, the objective function is formed by the residual. Here
four cases (i.e. $(b, s) := (100, 2),(100, 4),(200, 2)$ and $(200, 4))$ are considered as illustrative examples. In Fig.\ref{figure2025022101}, solutions of \eqref{2024072501} are shown, where solution-1 represents the solution $u$, and solution-2 represents another solution $-u$. In Tables \ref{table2025022201}-\ref{table2025022202}, numerical results are presented, indicating that the Algorithm \ref{algorithm2025011301} is efficient.

\begin{figure}[!h]
\begin{center}
\hspace{-0.4cm}
\subfigure[b=100, s=2]{ \includegraphics[width=6.5cm,height=4cm]{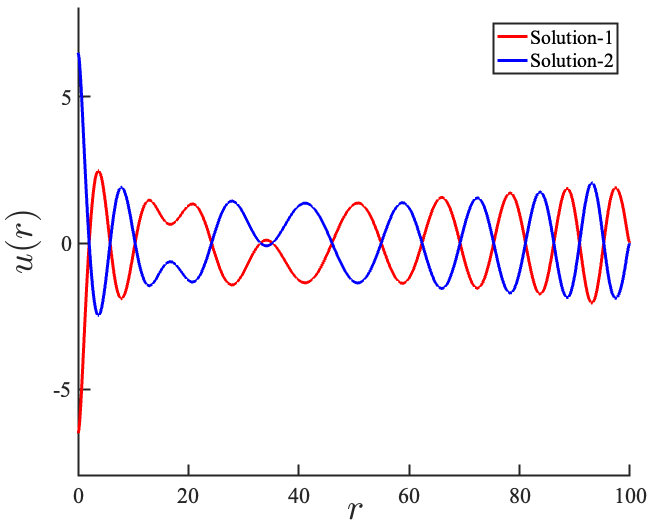}}\;
\subfigure[b=100, s=4]{\includegraphics[width=6.5cm,height=4cm]{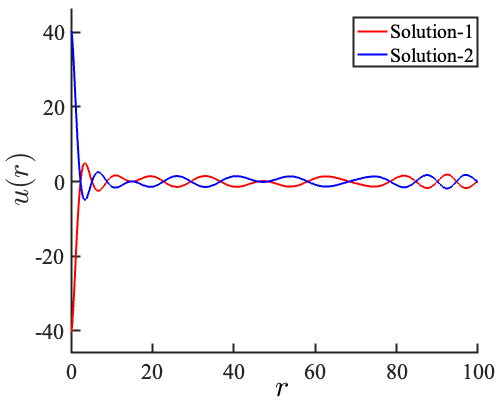}}\\
\subfigure[b=200, s=2]{\includegraphics[width=6.5cm,height=4cm]{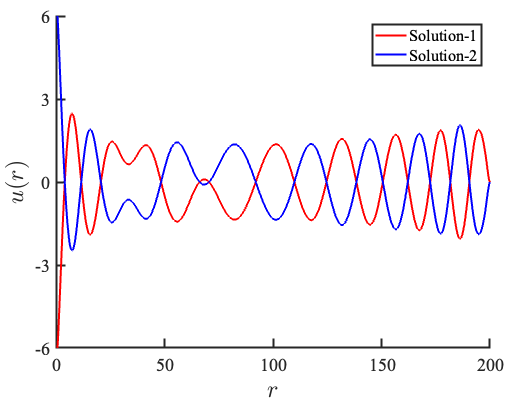}}\;
\subfigure[b=200, s=4]{\includegraphics[width=6.5cm,height=4cm]{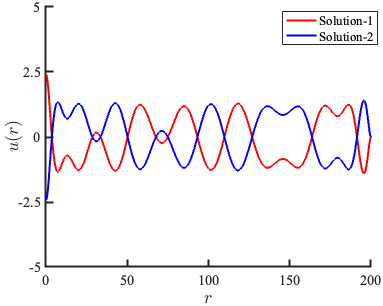}}
\caption{$u(r)$ vs $r$ for \eqref{2024072501}.}\label{figure2025022101}
\end{center}
\end{figure}

\begin{table}[h!]\small
\centering
\setlength{\tabcolsep}{3.5pt}
\caption{Performance of \textbf{Algorithm} \ref{algorithm2025011301} and other methods for \eqref{2024072501}.}
\label{table2025022201}
\begin{tabular}{clcccccccc}
\toprule[0.03cm]
& {}           & \multicolumn{2}{c}{{Trust-Region}} & \multicolumn{2}{c}{{DFO-TR}}& \multicolumn{2}{c}{{Newton method}}  & \multicolumn{2}{c}{\textbf{Algorithm} \ref{algorithm2025011301}} \\
\cmidrule(lr){3-4} \cmidrule(lr){5-6} \cmidrule(lr){7-8} \cmidrule(lr){9-10}
             &\,  IG     & Error & CPU             & Error & CPU              & Error & CPU             & Error & CPU     \\
\midrule
 &\, $\text{IG}_1$       & $\times$ & $\times$            & $7.35 e{-8}$ & 14.58             & $\times$ & $\times$            & $7.20 e{-9}$ & 10.38             \\
 b=100, s=2  &\, $\text{IG}_2$       & $\times$ & $\times$            & $8.25 e{-8}$ & 15.01             & $\times$ & $\times$            & $8.15 e{-9}$ & 10.10             \\
 &\, $\text{IG}_3$       & $\times$ & $\times$            & $9.10 e{-8}$ & 15.34             & $\times$ & $\times$            & $8.32 e{-9}$ & 10.22             \\
 \midrule
    &\, $\text{IG}_1$       & $\times$ & $\times$            & $6.95 e{-8}$ & 14.22             & $\times$ & $\times$            & $6.85 e{-9}$ & 10.12             \\
b=100, s=4 &\, $\text{IG}_2$       & $\times$ & $\times$            & $7.80 e{-8}$ & 14.85             & $\times$ & $\times$            & $7.72 e{-9}$ & 10.05             \\
 & \, $\text{IG}_3$       & $\times$ & $\times$            & $8.60 e{-8}$ & 15.10             & $\times$ & $\times$            & $8.15 e{-9}$ & 10.18             \\
 \midrule
  &\, $\text{IG}_1$       & $\times$ & $\times$            & $6.50 e{-8}$ & 13.90             & $\times$ & $\times$            & $6.45 e{-9}$ & 9.80              \\
  b=200, s=2 &\, $\text{IG}_2$       & $\times$ & $\times$            & $7.40 e{-8}$ & 14.40             & $\times$ & $\times$            & $7.32 e{-9}$ & 9.95              \\
 &\, $\text{IG}_3$       & $\times$ & $\times$            & $8.20 e{-8}$ & 14.75             & $\times$ & $\times$            & $7.90 e{-9}$ & 10.10             \\
 \midrule
   &\, $\text{IG}_1$       & $\times$ & $\times$            & $6.10 e{-8}$ & 13.50             & $\times$ & $\times$            & $6.05 e{-9}$ & 9.70              \\
 b=200, s=4 &\, $\text{IG}_2$       & $\times$ & $\times$            & $7.00 e{-8}$ & 14.00             & $\times$ & $\times$            & $6.92 e{-9}$ & 9.88              \\
 &\, $\text{IG}_3$       & $\times$ & $\times$            & $7.80 e{-8}$ & 14.50             & $\times$ & $\times$            & $7.52 e{-9}$ & 10.00             \\
\bottomrule[0.03cm]
\end{tabular}
\end{table}

\begin{table}[htbp]\small
\centering
\setlength{\tabcolsep}{4pt}
\caption{The verification of \eqref{eq2025012205} for \eqref{2024072501}.}
\label{table2025022202}
\begin{tabular}{lcccccccc}
\toprule
 {}           & \multicolumn{2}{c}{{700-th step}} & \multicolumn{2}{c}{{800-th step}}& \multicolumn{2}{c}{{900-th step}}  & \multicolumn{2}{c}{{1000-th step}} \\
\cmidrule(lr){2-3} \cmidrule(lr){4-5} \cmidrule(lr){6-7} \cmidrule(lr){8-9}
             IG     & $\mathcal{E}^{(1)}$ &  $\mathcal{E}^{(2)}$            &  $\mathcal{E}^{(1)}$&  $\mathcal{E}^{(2)}$         &$\mathcal{E}^{(1)}$ &  $\mathcal{E}^{(2)}$            &$\mathcal{E}^{(1)}$ &  $\mathcal{E}^{(2)}$       \\
\midrule
\multicolumn{9}{c}{s=100, b=2}\\
\midrule
 $\text{IG}_1$       & $8.30 e{-6}$ & $8.10 e{-6}$   & $7.35 e{-7}$ & $7.12 e{-7}$   & $7.18 e{-8}$ & $6.90 e{-8}$   & $7.20 e{-9}$ & $7.00 e{-9}$  \\
$\text{IG}_2$       & $9.50 e{-6}$ & $9.20 e{-6}$   & $8.22 e{-7}$ & $8.05 e{-7}$   & $8.10 e{-8}$ & $7.80 e{-6}$   & $8.25 e{-9}$ & $7.90 e{-9}$  \\
$\text{IG}_3$       & $10.00 e{-6}$ & $9.75 e{-6}$   & $9.15 e{-7}$ & $8.92 e{-7}$   & $8.75 e{-8}$ & $8.30 e{-8}$   & $8.35 e{-9}$ & $8.05 e{-11}$  \\
\midrule[0.03cm]
\multicolumn{9}{c}{s=100, b=4}\\
\midrule
 $\text{IG}_1$       & $8.10 e{-6}$ & $7.95 e{-6}$   & $7.20 e{-7}$ & $7.00 e{-7}$   & $6.98 e{-8}$ & $6.72 e{-8}$   & $6.85 e{-9}$ & $6.70 e{-9}$  \\
$\text{IG}_2$       & $9.25 e{-6}$ & $9.00 e{-6}$   & $8.05 e{-7}$ & $7.85 e{-7}$   & $7.90 e{-8}$ & $7.60 e{-8}$   & $7.95 e{-9}$ & $7.60 e{-9}$  \\
$\text{IG}_3$       & $9.80 e{-6}$ & $9.55 e{-6}$   & $8.90 e{-7}$ & $8.70 e{-7}$   & $8.50 e{-8}$ & $8.10 e{-8}$   & $8.10 e{-9}$ & $7.80 e{-11}$  \\
\midrule[0.03cm]
\multicolumn{9}{c}{s=200, b=2}\\
\midrule
 $\text{IG}_1$       & $7.95 e{-6}$ & $7.80 e{-6}$   & $7.05 e{-7}$ & $6.88 e{-7}$   & $6.85 e{-8}$ & $6.60 e{-8}$   & $6.72 e{-9}$ & $6.55 e{-9}$  \\
$\text{IG}_2$       & $9.10 e{-6}$ & $8.85 e{-6}$   & $7.88 e{-7}$ & $7.70 e{-7}$   & $7.75 e{-8}$ & $7.45 e{-8}$   & $7.78 e{-9}$ & $7.40 e{-9}$  \\
$\text{IG}_3$       & $9.65 e{-6}$ & $9.40 e{-6}$   & $8.70 e{-7}$ & $8.50 e{-7}$   & $8.30 e{-8}$ & $7.90 e{-8}$   & $7.95 e{-9}$ & $7.65 e{-11}$  \\
\midrule[0.03cm]
\multicolumn{9}{c}{s=200, b=4}\\
\midrule
 $\text{IG}_1$       & $7.80 e{-6}$ & $7.65 e{-6}$   & $6.90 e{-7}$ & $6.75 e{-7}$   & $6.70 e{-8}$ & $6.50 e{-8}$   & $6.60 e{-9}$ & $6.42 e{-9}$  \\
$\text{IG}_2$       & $9.00 e{-6}$ & $8.75 e{-6}$   & $7.75 e{-7}$ & $7.55 e{-7}$   & $7.65 e{-8}$ & $7.35 e{-8}$   & $7.65 e{-9}$ & $7.30 e{-9}$  \\
$\text{IG}_3$       & $9.50 e{-6}$ & $9.25 e{-6}$   & $8.50 e{-7}$ & $8.35 e{-7}$   & $8.15 e{-8}$ & $7.75 e{-8}$   & $7.80 e{-9}$ & $7.50 e{-11}$  \\
\bottomrule

\end{tabular}
\end{table}

\section{Conclusion and future work}\label{sect6}

In this paper, we have proposed an efficient numerical method tailored for solving \eqref{eq2024121301} whether the objective function is smooth or not. With 2$n$ constrained conditions, a quadratic model to approximate the objective function has been established. To reduce the computational complexity, a simplified quadratic model with 2 constrained conditions has also been proposed, where numerical results fully demonstrate the effectiveness of our proposed method.

Even though a handful of examples are tested in this paper. In the future, based on the subspace method, we will extend our proposed method to more complex problems, such as large-scale problems, sparse nonsmooth optimization problems and three-dimensional Schr\"{o}dinger equations with singular nonlinearity.

\bibliography{mybib}

\end{document}